\documentclass{article}

\usepackage[affil-it]{authblk}
\usepackage{amsmath}
\allowdisplaybreaks
\usepackage{mathrsfs}
\usepackage{mathtools}
\usepackage{amssymb} 
\usepackage{xspace}
\usepackage{amsthm}
\usepackage{enumitem}
\usepackage{algorithm}
\usepackage[noend]{algpseudocode}
\usepackage{tikz}
\usetikzlibrary{hobby}
\usepackage{subcaption}
\usetikzlibrary{positioning}
\usepackage{microtype}
\usepackage{hyperref}
\usepackage{booktabs}
\usepackage{multirow}
\usepackage[dvipsnames]{xcolor}
\usepackage{todonotes}
\usepackage{setspace}
\usepackage[margin=2.54cm]{geometry}
\usepackage{natbib}

\usepackage[normalem]{ulem}

\newcommand{\ZZ}{{\mathbb Z}}
\newcommand{\PP}{{\mathcal P}}
\newcommand{\NP}{{\mathcal{NP}}}
\newcommand{\FF}{{\mathcal F}}
\newcommand{\FORT}{{\mathscr F}}
\newcommand{\CC}{{\mathscr C}}
\newcommand{\FPSIP}{FPS-IP\xspace}
\newcommand{\FPSIPF}[1]{FPS-IP(#1)}
\newcommand{\EFPSIP}{EFPS-IP\xspace}
\newcommand{\BRIIP}{BRI-IP\xspace}
\newcommand{\JOVIP}{JOV-IP\xspace}
\newcommand{\FORTIP}{FORT-IP\xspace}

\newtheorem{proposition}{Proposition}
\newtheorem{corollary}{Corollary}
\newtheorem{lemma}{Lemma}
\newtheorem{remark}{Remark}
\newtheorem{definition}{Definition}

\title{Capacitated power dominating set problem: a solution approach based on forbidden propagation sets}

\author[1,2,*]{Mauro Lucci}
\author[3]{Diego Delle Donne}
\author[1,2]{Mariana Escalante}

\affil[1]{Depto. de Matemática (FCEIA), Universidad Nacional de Rosario, Av. Pellegrini 250, Rosario 2000, Argentina}
\affil[2]{CONICET, Ocampo y Esmeralda, Rosario 2000, Argentina}
\affil[3]{ESSEC Business School, Department of Information Systems, Data Analytics and Operations, 3 avenue Bernard Hirsch, Cergy 95021, France}
\affil[*]{Corresponding author: Mauro Lucci, mlucci@fceia.unr.edu.ar}

\date{}

\begin{document}

\maketitle

\begin{abstract}
The optimal placement of measurement devices in electrical power systems is commonly modeled through the power dominating set problem.
However, in real-world applications, these devices have limited capacities, leading to a capacitated variant of the problem that has received little attention in the literature.
In this work, we introduce forbidden propagation sets, novel combinatorial structures that cannot occur simultaneously in any feasible solution.
This notion enables a new class of integer linear programming formulations.
They combine infection-based variables with exponentially many constraints, while avoiding big-$M$ constraints.
We derive structural properties, valid inequalities, and redundancy-breaking constraints, and design an efficient lazy-separation procedure based on cycle detection.
Computational experiments on benchmark instances with up to 14,000 vertices show that the proposed method achieves an average execution-time improvement of 1.7x over existing approaches adapted from the literature.
Moreover, the results indicate that performance depends not only on network size, but also on capacities.
\end{abstract}

\section{Introduction}

Monitoring electric power systems is a fundamental control mechanism to ensure secure and reliable operations.
In recent decades, this area has received considerable attention, motivated by the widespread usage of measurement devices with high sampling rate and accuracy, known as \emph{phasor measurement units} (PMUs), and by the global interest in smart grids in the context of clean energies \citep{Joshi2021}.
The power dominating set problem (PDS) has emerged as a fundamental model for determining the minimum number of measurement devices, together with their optimal locations, required to monitor the entire network 
\citep{Baldwin1993}.
However, real-world applications often involve additional limitations and requirements that are not adequately captured by standard models.
Representative examples include limitations in the number of channels \citep{Korkali2009}, connectivity settings \citep{Fan2012}, and fault-tolerant strategies \citep{Almasabi2019}.
These considerations highlight the need for more realistic models that account for such practical constraints.

As a graph optimization problem, PDS was first introduced by \citet{Haynes2002}.
However, the definition most commonly used in the literature is due to \citet{Brueni2005}, who proposed a simplified but equivalent problem.
The objective is to find a subset $S$ of vertices with minimum cardinality, representing PMU locations, such that all vertices of the graph can be monitored through the sequential application of two rules.
First, each vertex in $S$ monitors its closed neighborhood, as in the well-known dominating set problem.
Second, a monitored vertex can monitor one of its neighbors if all its remaining neighbors are already monitored, a mechanism commonly referred to as propagation.
Furthermore, the last rule can only be applied at vertices satisfying the \emph{zero-injection} property, meaning they have neither generation nor load.
The main drawback of this standard model is that the first rule assumes that a PMU has sufficient channels (capacity) to monitor all of its neighbors, which is unrealistic in practice.
Only a few studies address this capacitated variant, called \emph{capacitated power dominating set} problem (CPDS), which introduces additional computational challenges: in contrast to PDS, it is necessary not only to decide the PMU locations, but also which neighbors each PMU monitors.

Regarding computational hardness, PDS remains NP-hard even when restricted to bipartite or chordal graphs \citep{Haynes2002}, as well as planar or split graphs \citep{Guo2008}.
On the other hand, linear-time algorithms have been developed for several classes, including trees \citep{Haynes2002}, block graphs \citep{Xu2006}, circular-arc graphs \citep{Liao2013}, and graphs of bounded treewidth \citep{Guo2008}.
Concerning exact algorithms for arbitrary graphs, a moderately exponential-time algorithm with running time $\mathcal{O}^*(1.7548^n)$ was theoretically proposed in \citet{Binkele-Raible2012}.
From a practical perspective, the most successful exact approaches are based on \emph{integer linear programming} (ILP).
Two main modeling paradigms have been explored.
Infection-based models explicitly track the application of the rules over time, as in \citet{Brimkov2019, Jovanovic2020}.
These formulations are compact, but use integer variables and big-$M$ constraints to model the timestep at which each vertex is monitored.
In contrast, fort-based models rely on a hitting-set characterization, where solutions must intersect the neighborhoods of certain combinatorial structures known as \emph{forts} \citep{Bozeman2019}.
These formulations use a single binary variable per vertex, but involve exponentially many constraints instead.
A fort-based algorithm with dynamic generation of constraints is the current state of the art, which can solve benchmark instances with up to 82,000 vertices in a few minutes \citep{Blasius2024}.

As discussed previously, the incorporation of capacities has received little attention in the literature.
\citet{Carvalho2018} propose single-level and bilevel ILP formulations for both PDS and CPDS .
However, their model admits an additional propagation rule that allows a vertex to propagate to itself whenever all of its neighbors are already monitored.
To capture the capacity requirements, the authors introduce new variables associated with the graph edges in order to control the individual application of the 
domination rule.
From a computational perspective, this approach exhibits limited scalability, as the two largest instances solved contain 300 and 2007 vertices.

For these reasons, this work aims to advance the development of exact algorithms for CPDS, with a strong focus on practical applicability.
To this end, we introduce the notion of \emph{forbidden propagation sets}, a novel combinatorial structure grouping propagations that cannot be applied together because they create cyclic temporal precedences.
This concept enables ILP formulations that combine some infection-based variables with exponentially many constraints, as in fort-based models, while avoiding big-$M$ constraints that can severely weaken linear relaxation bounds.
To our knowledge, this approach has not previously been considered in the PDS literature or related propagation problems, such as the zero forcing problem \citep{Bozeman2019}. 

Computational experiments show that the proposed approach outperforms existing ILP formulations adapted from the literature.
Some preliminary results of this research were presented in a short work \citep{Lucci2025}.
However, new structural insights enable for more tightened formulations and competitive algorithms.
This work is organized as follows.
Section \ref{sec.cpds} formalizes the CPDS and introduces preliminary notation and definitions.
Section \ref{sec.fps} defines forbidden propagation sets and some structural properties.
Section \ref{sec.fpsip} presents ILP formulations for CPDS and develops an efficient lazy-separation procedure.
Section \ref{sec.formulations.from.literature} adapts existing ILP formulations from the PDS literature to CPDS.
Section \ref{sec.experiments} reports and analyzes the computational results.
Finally, Section \ref{sec.conclusions.future.work} exhibits the conclusions and future work.
The main contributions of this research can be summarized as follows.

\begin{itemize}
\item We introduce \emph{forbidden propagation sets}, a novel concept in the power dominating set literature that captures infeasible combinations of propagation rules through cyclic temporal precedences.
\item We propose a new class of ILP formulations that rely on infection-based variables, but replace traditional big-$M$ constraints with an exponential family of inequalities associated with forbidden propagation sets.
\item We design a polynomial-time lazy-constraint separation that takes advantage of the derived structural results.
\item Computational experiments on benchmark instances for the capacitated power dominating set problem demonstrate the effectiveness of the proposed approach, compared to other ILP formulations adapted from the literature.
\end{itemize}

\section{Capacitated power dominating set problem}\label{sec.cpds}

In this work, we use standard definitions and notation from graph theory.
Given a vertex $v$, $\delta(v)$ denotes its degree, $N(v)$ its set of neighbors, and $N[v] = N(v) \cup \{v\}$.
In a digraph $D$, a (simple) cycle $C$ is defined as a sequence $((v_1,v_2), (v_2,v_3), \ldots, (v_n,v_1))$ of arcs of $D$, where $v_i \neq v_j$ for all $i,j \in \{1,\dots,n\}$. In this case, the cycle has length $|C|=n$ and we may refer to it as an $n$-cycle.
We say that an arc $(v_i,v_j) \in C$ if $(v_i,v_j)$ appears in the sequence that defines $C$.
A chord of $C$ is any arc of $D$ that is not in $C$ and whose endpoints are both vertices of $C$.
We denote by $\CC(D)$ the set of cycles of $D$.

Let $G=(V,E)$ be the graph that represents the electric power system, where the vertices are the buses and the edges are the transmission lines. 
Let $V_P \subseteq V$ denote the subset of zero-injection vertices.
The PMU locations are represented by a set $S \subseteq V$.
To encode the decisions arising from the channel limitation, we introduce a function $\rho: S \to \PP(V)$ such that, for each $v \in S$, $\rho(v) \subseteq N(v)$.
Since PMU locations can be recovered from the domain of $\rho$, we usually denote $S_{\rho} = \text{Dom}(\rho)$.
Monitoring obeys the following two rules.
Each $u \in S_{\rho}$ monitors itself and all neighbors in $\rho(u)$.
Moreover, whenever a monitored vertex $u$ has at least $\delta(u) - 1$ monitored neighbors, its remaining neighbor also becomes monitored, and this rule can be applied iteratively.
Formally, we define the \emph{monitored set} $M(\rho)$ of $\rho$ as the smallest subset of $V$ satisfying:
\begin{description}
    \item[Domination rule (DR).] If $u \in S_{\rho}$ and $v \in (\rho(u) \cup \{u\})$, then $v \in M(\rho)$.
    \item[Propagation rule (PR).] If $u \in V_P \cap M(\rho)$, $v \in N(u)$ and $N(u) \setminus \{v\} \subseteq M(\rho)$, then $v \in M(\rho)$.
\end{description}
Let $k \in \ZZ_+$ denote the number of channels available (the capacity) in each PMU.
When $|\rho(v)| \leq k$ for all $v \in S_{\rho}$, $\rho$ is referred to as a \emph{$k$-capacitated function} of $(G,V_P)$. If in addition, $M(\rho) = V$, then we say that $\rho$ is a \emph{$k$-capacitated power dominating function}.

Figure \ref{fig.pds} illustrates the step-by-step calculation of the monitored set $M(\rho)$ for $\rho$ with $S_{\rho} = \{a\}$ and $\rho(a) = \{b,d\}$, where rectangular vertices have the zero-injection property.
In Figure \ref{fig.pds1}, the domination rule (DR) is applied at $a$ to monitor $a$, $b$, and $d$.
At this stage, $a$ and $d$ are monitored and each has exactly one unmonitored neighbor, so the propagation rule (PR) can be applied at both vertices.
In contrast, this rule cannot be applied at $b$, because it has two unmonitored neighbors, nor at $c$, since it is unmonitored.
In Figure \ref{fig.pds2}, the propagation rule (PR) is applied at $a$ to monitor $e$, and at $d$ to monitor $c$.
It is worth noting that both rules (DR) and (PR) are applied at $a$, a behavior that does not arise in the classical PDS.
At this point, $b$ and $c$ become available for the application of the propagation rule (PR), highlighting the temporal aspect inherent to the sequential application of the rules.
In Figure \ref{fig.pds3}, the propagation rule (PR) is applied at $b$ to monitor $f$ and at $c$ to monitor $g$, resulting $M(\rho) = V$.
This demonstrates that $\rho$ is a 2-capacitated power dominating function.

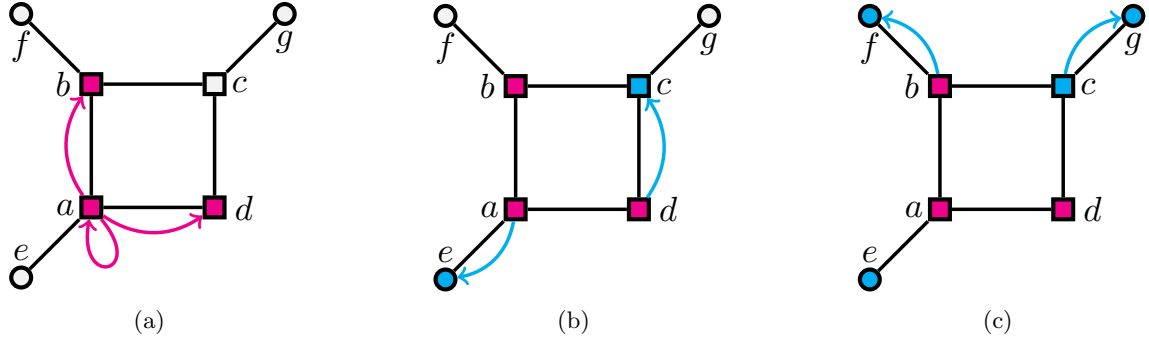
\begin{figure}

\begin{subfigure}{0.32\textwidth}
\centering
\resizebox{0.8\linewidth}{!}{
\begin{tikzpicture}[
every node/.style={circle, draw=black, fill=black!5, very thick, minimum size=2mm, inner sep=0},
every path/.style={line width=1pt},
scale=0.8
]
\node[label=left:$a$, fill=Magenta, rectangle] (a) {};
\node[label=left:$b$, above=of a, fill=Magenta, rectangle] (b) {};
\node[label=right:$c$, right=of b, rectangle] (c) {};
\node[label=right:$d$, right=of a, fill=Magenta, rectangle] (d) {};
\node[label=above:$e$, below left= 0.5cm and 0.5cm of a] (e) {};
\node[label=below:$f$, above left= 0.5cm and 0.5cm of b] (f) {};
\node[label=below:$g$, above right= 0.5cm and 0.5cm of c] (g) {};
\draw (a) -- (b) -- (c) -- (d) -- (a);
\draw (a) -- (e);
\draw (c) -- (g);
\draw (b) -- (f);
\draw[ ->, Magenta] (a) to [out=310,in=260,looseness=15] (a);
\draw[ ->, Magenta] (a) to [bend left=35] (b);
\draw[ ->, Magenta] (a) to [bend right=35] (d);
\end{tikzpicture}
}
\caption{}
\label{fig.pds1}
\end{subfigure}
\hfill
\begin{subfigure}{0.32\textwidth}
\centering
\resizebox{0.8\linewidth}{!}{
\begin{tikzpicture}[
every node/.style={circle, draw=black, fill=black!5, very thick, minimum size=2mm, inner sep=0},
every path/.style={line width=1pt},
scale=0.8
]
\node[label=left:$a$, fill=Magenta, rectangle] (a) {};
\node[label=left:$b$, above=of a, fill=Magenta, rectangle] (b) {};
\node[label=right:$c$, right=of b, fill=Cyan, rectangle] (c) {};
\node[label=right:$d$, right=of a, fill=Magenta, rectangle] (d) {};
\node[label=above:$e$, below left= 0.5cm and 0.5cm of a, fill=Cyan] (e) {};
\node[label=below:$f$, above left= 0.5cm and 0.5cm of b] (f) {};
\node[label=below:$g$, above right= 0.5cm and 0.5cm of c] (g) {};
\draw (a) -- (b) -- (c) -- (d) -- (a);
\draw (a) -- (e);
\draw (c) -- (g);
\draw (b) -- (f);
\draw[ ->, Cyan] (a) to [bend left=35] (e);
\draw[ ->, Cyan] (d) to [bend right=35] (c);
\end{tikzpicture}
}
\caption{}
\label{fig.pds2}
\end{subfigure}
\hfill
\begin{subfigure}{0.32\textwidth}
\centering
\resizebox{0.8\linewidth}{!}{
\begin{tikzpicture}[
every node/.style={circle, draw=black, fill=black!5, very thick, minimum size=2mm, inner sep=0},
every path/.style={line width=1pt},
scale=0.8
]
\node[label=left:$a$, fill=Magenta, rectangle] (a) {};
\node[label=left:$b$, above=of a, fill=Magenta, rectangle] (b) {};
\node[label=right:$c$, right=of b, fill=Cyan, rectangle] (c) {};
\node[label=right:$d$, right=of a, fill=Magenta, rectangle] (d) {};
\node[label=above:$e$, below left= 0.5cm and 0.5cm of a, fill=Cyan] (e) {};
\node[label=below:$f$, above left= 0.5cm and 0.5cm of b, fill=Cyan] (f) {};
\node[label=below:$g$, above right= 0.5cm and 0.5cm of c, fill=Cyan] (g) {};
\draw (a) -- (b) -- (c) -- (d) -- (a);
\draw (a) -- (e);
\draw (c) -- (g);
\draw (b) -- (f);
\draw[ ->, Cyan] (b) to [bend right=35] (f);
\draw[ ->, Cyan] (c) to [bend left=35] (g);
\end{tikzpicture}
}
\caption{}
\label{fig.pds3}
\end{subfigure}

\caption{Example of the step-by-step calculation of the monitored set. Colors indicate the rule applied: magenta for the domination rule (DR) and cyan for the propagation rule (PR).}
\label{fig.pds}
\end{figure}

This work focuses on the following optimization problem.

\medskip

\noindent
\textbf{Capacitated Power Dominating Set problem (CPDS)}\\
\textbf{Instance:} A graph $G = (V,E)$,  a subset $V_P \subseteq V$, and an integer $k \in \ZZ_+$.\\
\textbf{Goal:} Find a $k$-capacitated power dominating function $\rho$ of $(G,V_P)$ such that $|S_{\rho}|$ is minimum.

\medskip

In particular, CPDS restricted to instances where $k$ is at least the maximum degree of $G$ is equivalent to PDS, where the goal reduces to finding the set $S \subseteq V$ and defining $\rho(u) = N(u)$ for every $u \in S$.
In this setting, a set $S$ such that $M(\rho) = V$ is known as a \emph{power dominating set}.
Conversely, its restriction to instances with $k = 0$ is equivalent to the zero-forcing problem \citep{Bozeman2019}.
Consequently, CPDS is $\NP$-hard.
Since CPDS can be solved independently on each connected component of $G$, it suffices to consider instances where $G$ is connected.

\subsection{Propagations and precedence digraph}

Given a $k$-capacitated function $\rho$, a \emph{calculation} of $M(\rho)$ consists of a finite sequence of applications of rules (DR) and (PR) in a specified order such that, at the end, neither rule can be further applied to monitor any remaining unmonitored vertices.
A calculation of $M(\rho)$ is \emph{proper} when each applied rule increases the size of the monitored set by exactly one, i.e., when the calculation has no \emph{redundant} applications of the rules.

The propagation rule (PR) can potentially be applied at any vertex $u \in V_P$ to monitor a neighbor $v \in N(u)$, provided that the preconditions of the rule are satisfied.
We represent each such possibility by the pair $(u,v)$, referred to as a \emph{propagation}, more precisely, an \emph{outgoing} propagation from $u$ and an \emph{incoming} propagation to $v$. 
For convenience, we denote by $A_P = \{(u,v): u \in V_P,\ v \in N(u)\}$ the set of all potential propagations.
For example, the set $A_P$ associated with the graph in Figure \ref{fig.pds} is given by the cyan propagations shown in the left graph of Figure \ref{fig.digraph.left}.
The propagations actually applied in the calculation of $M(\rho)$ form a subset of $A_P$; for instance, those used in the example in Figure \ref{fig.pds} correspond to the cyan propagations shown in the left graph of Figure \ref{fig.digraph.right}.

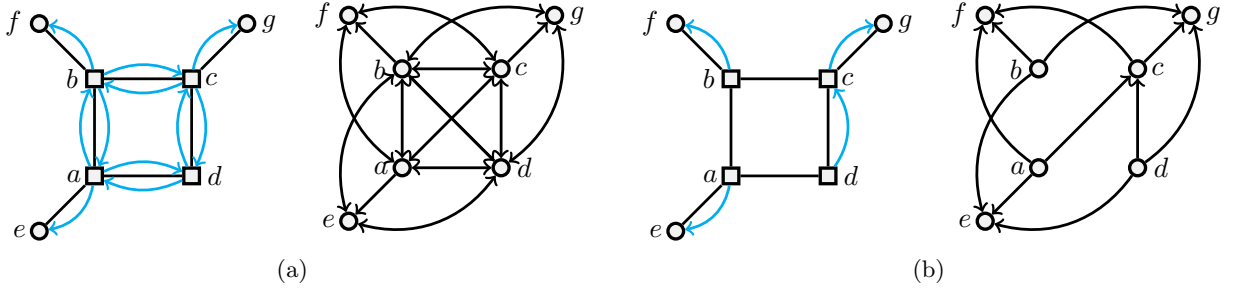
\begin{figure}

\begin{subfigure}{0.49\textwidth}
\centering
\resizebox{0.49\linewidth}{!}{
\begin{tikzpicture}[
every node/.style={circle, draw=black, fill=black!5, very thick, minimum size=2mm, inner sep=0},
every path/.style={line width=1pt},
scale=0.8
]
\node[label=left:$a$, rectangle] (a) {};
\node[label=left:$b$, above=of a, rectangle] (b) {};
\node[label=right:$c$, right=of b, rectangle] (c) {};
\node[label=right:$d$, right=of a, rectangle] (d) {};
\node[label=left:$e$, below left= 0.5cm and 0.5cm of a] (e) {};
\node[label=left:$f$, above left= 0.5cm and 0.5cm of b] (f) {};
\node[label=right:$g$, above right= 0.5cm and 0.5cm of c] (g) {};
\draw (a) -- (b) -- (c) -- (d) -- (a);
\draw (a) -- (e);
\draw (c) -- (g);
\draw (b) -- (f);
\draw[ ->, Cyan] (a) to [bend left=25] (b);
\draw[ ->, Cyan] (b) to [bend left=25] (a);
\draw[ ->, Cyan] (b) to [bend left=25] (c);
\draw[ ->, Cyan] (c) to [bend left=25] (b);
\draw[ ->, Cyan] (c) to [bend left=25] (d);
\draw[ ->, Cyan] (d) to [bend left=25] (c);
\draw[ ->, Cyan] (d) to [bend left=25] (a);
\draw[ ->, Cyan] (a) to [bend left=25] (d);
\draw[ ->, Cyan] (a) to [bend left=35] (e);
\draw[ ->, Cyan] (b) to [bend right=35] (f);
\draw[ ->, Cyan] (c) to [bend left=35] (g);
\end{tikzpicture}
}
\resizebox{0.49\linewidth}{!}{
\begin{tikzpicture}[
every node/.style={circle, draw=black, fill=black!5, very thick, minimum size=2mm, inner sep=0},
every path/.style={line width=1pt},
scale=0.8
]
\node[label=left:$a$] (a) {};
\node[label=left:$b$, above=of a] (b) {};
\node[label=right:$c$, right=of b] (c) {};
\node[label=right:$d$, right=of a] (d) {};
\node[label=left:$e$, below left= 0.5cm and 0.5cm of a] (e) {};
\node[label=left:$f$, above left= 0.5cm and 0.5cm of b] (f) {};
\node[label=right:$g$, above right= 0.5cm and 0.5cm of c] (g) {};
\draw[ ->] (a) to (e);
\draw[<->] (b) to [bend right=35] (e);
\draw[<->] (d) to [bend left=35] (e);
\draw[ ->] (b) to (f);
\draw[<->] (a) to [bend left=35] (f);
\draw[<->] (c) to [bend right=35] (f);
\draw[ ->] (c) to (g);
\draw[<->] (d) to [bend right=35] (g);
\draw[<->] (b) to [bend left=35] (g);

\draw[<->] (a) to (b);
\draw[ ->] (d) to (b);
\draw[<->] (c) to (a);
\draw[<->] (b) to (c);
\draw[<->] (c) to (d);
\draw[ ->] (b) to (d);
\draw[<->] (a) to (d);
\end{tikzpicture}
}
\caption{}
\label{fig.digraph.left}
\end{subfigure}
\hfill
\begin{subfigure}{0.49\textwidth}
\centering
\resizebox{0.49\linewidth}{!}{
\begin{tikzpicture}[
every node/.style={circle, draw=black, fill=black!5, very thick, minimum size=2mm, inner sep=0},
every path/.style={line width=1pt},
scale=0.8
]
\node[label=left:$a$, rectangle] (a) {};
\node[label=left:$b$, above=of a, rectangle] (b) {};
\node[label=right:$c$, right=of b, rectangle] (c) {};
\node[label=right:$d$, right=of a, rectangle] (d) {};
\node[label=left:$e$, below left= 0.5cm and 0.5cm of a] (e) {};
\node[label=left:$f$, above left= 0.5cm and 0.5cm of b] (f) {};
\node[label=right:$g$, above right= 0.5cm and 0.5cm of c] (g) {};
\draw (a) -- (b) -- (c) -- (d) -- (a);
\draw (a) -- (e);
\draw (c) -- (g);
\draw (b) -- (f);
\draw[ ->, Cyan] (a) to [bend left=35] (e);
\draw[ ->, Cyan] (d) to [bend right=35] (c);
\draw[ ->, Cyan] (b) to [bend right=35] (f);
\draw[ ->, Cyan] (c) to [bend left=35] (g);
\end{tikzpicture}
}
\resizebox{0.49\linewidth}{!}{
\begin{tikzpicture}[
every node/.style={circle, draw=black, fill=black!5, very thick, minimum size=2mm, inner sep=0},
every path/.style={line width=1pt},
scale=0.8
]
\node[label=left:$a$] (a) {};
\node[label=left:$b$, above=of a] (b) {};
\node[label=right:$c$, right=of b] (c) {};
\node[label=right:$d$, right=of a] (d) {};
\node[label=left:$e$, below left= 0.5cm and 0.5cm of a] (e) {};
\node[label=left:$f$, above left= 0.5cm and 0.5cm of b] (f) {};
\node[label=right:$g$, above right= 0.5cm and 0.5cm of c] (g) {};
\draw[ ->] (a) to (e);
\draw[ ->] (d) to [bend left=35] (e);
\draw[ ->] (b) to [bend right=35] (e);
\draw[ ->] (d) to (c);
\draw[ ->] (a) to (c);
\draw[ ->] (b) to (f);
\draw[ ->] (c) to [bend right=35] (f);
\draw[ ->] (a) to [bend left=35] (f);
\draw[ ->] (c) to (g);
\draw[ ->] (d) to [bend right=35] (g);
\draw[ ->] (b) to [bend left=35] (g);
\end{tikzpicture}
}
\caption{}
\label{fig.digraph.right}
\end{subfigure}

\caption{Two pairs of propagation sets (in cyan) and their corresponding precedence digraphs. The leftmost one is the precedence digraph $D$ associated with $A_P$.}
\label{fig.digraph}
\end{figure}

Applying the propagation rule (PR) imposes temporal precedences, requiring that certain vertices be monitored before others.
For example, in Figure \ref{fig.pds2}, the propagation $(a,e)$ requires vertices $a,b,$ and $d$ to already be monitored.
Specifically, a propagation $(u,v) \in A_P$ can be applied only if every vertex $w \in N[u] \setminus \{v\}$ is monitored before $v$.
Such a requirement is represented by the pair $(w,v)$, which we refer to as a \emph{precedence}.
Observe that all precedences arising from a given propagation share the target vertex with that propagation.

The connection between propagations and precedences is formalized through the following mapping.
Let $\psi$ be the function that assigns to each propagation $(u,v) \in A_P$ the set of precedences imposed by its application, formally $\psi(u,v) = \{(w,v): w \in N[u] \setminus \{v\}\}$.

Given $R \subseteq A_P$, define $\psi(R) = \bigcup_{(u,v) \in R} \psi(u,v)$.
The \emph{precedence digraph} imposed by the propagations in $R$ is then $D_R = (V, \psi(R))$.
For simplicity, we write $D_{A_P} = D$. 
Note that $D$ denotes the precedence digraph associated with all potential propagations in $A_P$, whereas $D_R$ denotes the digraph associated with the subset $R$ of propagations.
Clearly, each $D_R$ is a subgraph of $D$.
For example, the digraph on the right in Figure \ref{fig.digraph.left} illustrates the precedence digraph $D$ associated with $A_P$, where pairs of arcs in opposite directions are represented by a double-headed arrow.
In contrast, the digraph on the right in Figure \ref{fig.digraph.right} shows the precedence digraph $D_R$ corresponding to the subset $R$ of propagations actually applied in the calculation of $M(\rho)$ in the example of Figure \ref{fig.pds}.

For convenience, we also define an inverse mapping of $\psi$.
Let $\varphi$ be the function that assigns to each precedence $(w,v) \in \psi(A_P)$ (arc of $D$) the set of propagations of $A_P$ that impose such a precedence, i.e., $\varphi(w,v) = \{(u,v) \in A_P: w \in N[u] \setminus \{v\} \}$.
Observe that all propagations imposing a given precedence share the target vertex with that precedence.
For example, in Figure \ref{fig.digraph.left}, $\varphi(a,b) = \{(a,b)\}$, $\varphi(b,d) = \{(a,d),(c,d)\}$.
For every propagation $p \in A_P$ and precedence $e \in \psi(A_P)$, $e \in \psi(p)$ if and only if $p \in \varphi(e)$.

In the following remark, we collect some preliminary results and introduce additional notation that will be used frequently.
\begin{remark}\label{remark.preliminary.results}
Let $R \subseteq A_P$.
For every arc $e$ of $D_R$, the definition of $D_R$ implies that $e$ is imposed by some propagation in $R$, i.e. $\varphi(e) \cap R \neq \emptyset$.
Furthermore, for any two arcs $e$ and $e'$ of $D_R$ with distinct target vertices, it holds that $\varphi(e) \cap \varphi(e') = \emptyset$, since the propagations in $\varphi(e)$ (resp. in $\varphi(e')$) have the same target vertex as $e$ (resp. as $e'$).

\noindent
We say that $R$ has no \emph{redundant incoming propagations} if no vertex of $V$ receives more than one incoming propagation from $R$.
In that case, for every arc $e$ of $D_R$, $|\varphi(e) \cap R| = 1$, since any two propagations in this set would share the same target vertex.
We denote this unique propagation by $\varphi_R(e)$.
\end{remark}

\section{Forbidden propagation sets}\label{sec.fps}

We introduce a novel combinatorial structure, referred to as \emph{forbidden propagation sets}, representing sets of propagations that cannot be applied simultaneously in a proper calculation of the monitored set.
The definition of these structures is motivated by the following result.

\begin{lemma}\label{lemma.DAG}
Let $\rho$ be a $k$-capacitated function.
If $R \subseteq A_P$ is the set of propagations applied by rule (PR) in a proper calculation of $M(\rho)$, then the precedence digraph $D_R$ contains no cycles.
\end{lemma}
\begin{proof}
Let $R \subseteq A_P$ be the set of propagations applied by rule (PR) in a proper calculation of $M(\rho)$.
Suppose that $D_R$ contains a cycle $C = (e_1, e_2, \dots, e_r)$.
Since the calculation of $M(\rho)$ is proper, it follows that $R$ has no redundant incoming propagations.
Hence, by Remark \ref{remark.preliminary.results}, each arc $e_i = (v_i, v_{i+1}) \in C$ is imposed by a unique propagation in $R$, namely $p_i = \varphi_R(e_i)$, with target $v_{i+1}$, and this propagation does not impose any other arc of $C$.

When $p_i$ is applied following the order of the calculation of $M(\rho)$, vertex $v_i$ has already been monitored.
Since the calculation of $M(\rho)$ is proper, monitoring $v_i$ is possible only if $p_{i-1}$ (indices taken modulo $r$), which is the propagation of $R$ that monitors $v_i$, is applied before $p_i$.
As this holds for all $i \in \{1,\ldots,r\}$, a cyclic contradiction follows.
\end{proof}

Lemma \ref{lemma.DAG} motivates defining forbidden structures for sets of propagations in a proper computation of $M(\rho)$.

\begin{definition}\label{def.FPS}
A subset $F \subseteq A_P$ is called a \emph{forbidden propagation set} (FPS) if the precedence digraph $D_F$ contains at least one  cycle, i.e. $\CC(D_F) \neq \emptyset$.
\end{definition}

The next key result follows directly and serves as the motivation for the ILP formulations introduced in Section \ref{sec.fpsip}.

\begin{remark}\label{remark.fps}
Let $\rho$ be a $k$-capacitated function and $F$ be an FPS.
Any proper calculation of $M(\rho)$ applies at most $|F| - 1$ propagations of $F$.
\end{remark}

Observe that if $F$ is an FPS, then every $\tilde F \subseteq A_P$ with $F \subseteq \tilde F$ is also an FPS, 
since every cycle in $\CC(D_F)$ also belongs to $\CC(D_{\tilde F})$.
Two distinct notions of minimality related to FPSs naturally arise.

\begin{definition}\label{def.minimal.FPS}
~ 
\begin{enumerate}
    \item Let $C \in \CC(D)$. We say that $F \subseteq A_P$ is an FPS that \emph{minimally imposes} $C$ if $C \in \CC(D_F)$ and, for every $p \in F$, $C \notin \CC(D_{F \setminus \{p\}})$. 
    \item We say that $F \subseteq A_P$ is a \emph{minimal FPS} if $F$ is an FPS and, for every $p \in F$, $F \setminus \{p\}$ is no longer an FPS.
\end{enumerate}
\end{definition}

It is immediate that if $F$ is a minimal FPS, then $F$ minimally imposes every $C \in \CC(D_F)$.
However, the converse does not necessarily hold, as illustrated in the following example.
Considering the graph in Figure \ref{fig.pds}, Figure \ref{fig.fps.1} depicts a set $F_1 = \{(a,b),(b,c),(c,d)\}$, which is an FPS that minimally imposes the cycle $C_1 = ((d,b),(b,c),(c,d))$, since $C_1 \in \CC(D_{F_1})$ and, for every $p \in F_1$, $F_1 \setminus \{p\}$ no longer imposes one of the arcs of $C_1$.
Nevertheless, $F_1$ is not a minimal FPS, because the set $F_2 = F_1 \setminus \{(b,c)\}$, depicted in Figure \ref{fig.fps.2}, imposes the cycle $C_2 = \{(d,b),(b,d)\}$, and hence is an FPS.
In fact, it is immediate that $F_2$ is a minimal FPS.

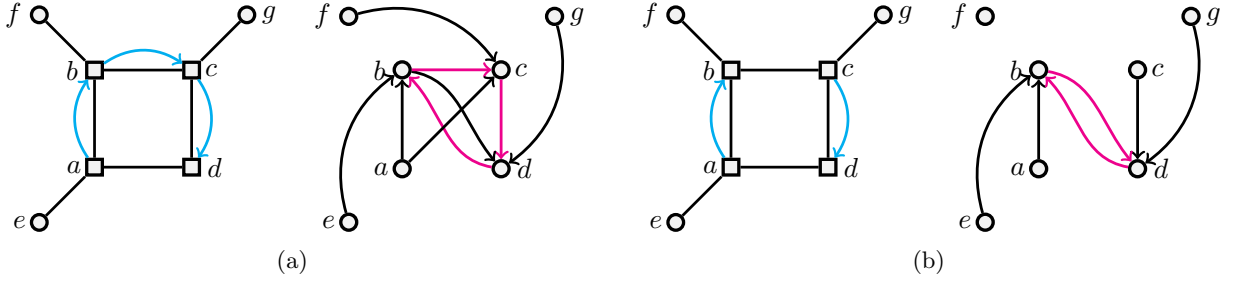
\begin{figure}

\begin{subfigure}{0.49\textwidth}
\centering
\resizebox{0.49\linewidth}{!}{
\begin{tikzpicture}[
every node/.style={circle, draw=black, fill=black!5, very thick, minimum size=2mm, inner sep=0},
every path/.style={line width=1pt},
scale=0.8
]
\node[label=left:$a$, rectangle] (a) {};
\node[label=left:$b$, above=of a, rectangle] (b) {};
\node[label=right:$c$, right=of b, rectangle] (c) {};
\node[label=right:$d$, right=of a, rectangle] (d) {};
\node[label=left:$e$, below left= 0.5cm and 0.5cm of a] (e) {};
\node[label=left:$f$, above left= 0.5cm and 0.5cm of b] (f) {};
\node[label=right:$g$, above right= 0.5cm and 0.5cm of c] (g) {};
\draw (a) -- (b) -- (c) -- (d) -- (a);
\draw (a) -- (e);
\draw (c) -- (g);
\draw (b) -- (f);
\draw[ ->, Cyan] (a) to [bend left=35] (b);
\draw[ ->, Cyan] (c) to [bend left=35] (d);
\draw[ ->, Cyan] (b) to [bend left=35] (c);
\end{tikzpicture}
}
\resizebox{0.49\linewidth}{!}{
\begin{tikzpicture}[
every node/.style={circle, draw=black, fill=black!5, very thick, minimum size=2mm, inner sep=0},
every path/.style={line width=1pt},
scale=0.8
]
\node[label=left:$a$] (a) {};
\node[label=left:$b$, above=of a] (b) {};
\node[label=right:$c$, right=of b] (c) {};
\node[label=right:$d$, right=of a] (d) {};
\node[label=left:$e$, below left= 0.5cm and 0.5cm of a] (e) {};
\node[label=left:$f$, above left= 0.5cm and 0.5cm of b] (f) {};
\node[label=right:$g$, above right= 0.5cm and 0.5cm of c] (g) {};
\draw[ ->] (a) to (b);
\draw[ ->] (e) to [bend left=35] (b);
\draw[ ->, Magenta] (d) to [out=170, in=315] (b);
\draw[ ->, Magenta] (c) to (d);
\draw[ ->] (b) to [out=350, in=130] (d);
\draw[ ->] (g) to [bend left=35] (d);
\draw[ ->, Magenta] (b) to (c);
\draw[ ->] (a) to (c);
\draw[ ->] (f) to [bend left=35] (c);
\end{tikzpicture}
}
\caption{}
\label{fig.fps.1}
\end{subfigure}
\hfill
\begin{subfigure}{0.49\textwidth}
\centering
\resizebox{0.49\linewidth}{!}{
\begin{tikzpicture}[
every node/.style={circle, draw=black, fill=black!5, very thick, minimum size=2mm, inner sep=0},
every path/.style={line width=1pt},
scale=0.8
]
\node[label=left:$a$, rectangle] (a) {};
\node[label=left:$b$, above=of a, rectangle] (b) {};
\node[label=right:$c$, right=of b, rectangle] (c) {};
\node[label=right:$d$, right=of a, rectangle] (d) {};
\node[label=left:$e$, below left= 0.5cm and 0.5cm of a] (e) {};
\node[label=left:$f$, above left= 0.5cm and 0.5cm of b] (f) {};
\node[label=right:$g$, above right= 0.5cm and 0.5cm of c] (g) {};
\draw (a) -- (b) -- (c) -- (d) -- (a);
\draw (a) -- (e);
\draw (c) -- (g);
\draw (b) -- (f);
\draw[ ->, Cyan] (a) to [bend left=35] (b);
\draw[ ->, Cyan] (c) to [bend left=35] (d);
\end{tikzpicture}
}
\resizebox{0.49\linewidth}{!}{
\begin{tikzpicture}[
every node/.style={circle, draw=black, fill=black!5, very thick, minimum size=2mm, inner sep=0},
every path/.style={line width=1pt},
scale=0.8
]
\node[label=left:$a$] (a) {};
\node[label=left:$b$, above=of a] (b) {};
\node[label=right:$c$, right=of b] (c) {};
\node[label=right:$d$, right=of a] (d) {};
\node[label=left:$e$, below left= 0.5cm and 0.5cm of a] (e) {};
\node[label=left:$f$, above left= 0.5cm and 0.5cm of b] (f) {};
\node[label=right:$g$, above right= 0.5cm and 0.5cm of c] (g) {};
\draw[ ->] (a) to (b);
\draw[ ->] (e) to [bend left=35] (b);
\draw[ ->, Magenta] (d) to [out=170, in=315] (b);
\draw[ ->] (c) to (d);
\draw[ ->, Magenta] (b) to [out=350, in=130] (d);
\draw[ ->] (g) to [bend left=35] (d);
\end{tikzpicture}
}
\caption{}
\label{fig.fps.2}
\end{subfigure}

\caption{Two pairs of FPSs (in cyan) and their precedence digraphs: (a) an FPS that minimally imposes the magenta cycle; (b) a minimal FPS properly contained in (a).}
\label{fig.fps}
\end{figure}

\subsection{Minimality in forbidden propagation sets}

We now focus on characterizing FPSs that minimally impose a given cycle, as well as minimal FPSs. Let us first present some intermediate results.

\begin{remark}\label{remark.FPS.minimally.impose.cycle}
Let $C \in \CC(D)$. If $F \subseteq A_P$ is an FPS that minimally imposes $C$, then, for all $p \in F$, there exists a precedence (arc) $e \in C$ imposed by $p$ and not imposed by any other propagation in $F$.
Formally, this means that $e \in \psi(p)$ and $e \notin \psi(p')$ for all $p' \in F \setminus \{p\}$, or equivalently, $\varphi(e) \cap F = \{p\}$. 
\end{remark}

We show that these minimal structures are free of redundancy.

\begin{lemma}\label{lemma.minimally.impose.implies.no.redundancy}
    Let $C \in \CC(D)$.
    If $F \subseteq A_P$ is an FPS that minimally imposes $C$, then $F$ has no redundant incoming propagations.
\end{lemma}
\begin{proof}
Let $F \subseteq A_P$ be an FPS that minimally imposes $C$.
Suppose that $F$ contains two propagations $p$ and $p'$ sharing the same target vertex $v$.
By Remark \ref{remark.FPS.minimally.impose.cycle}, there exist arcs $e, e' \in C$ such that $e \in \psi(p) \setminus \psi(p')$ and $e' \in \psi(p') \setminus \psi(p)$.
By Remark \ref{remark.preliminary.results}, both $e$ and $e'$ have target $v$.
Since all arcs of a cycle have distinct targets, it follows that $e = e'$, which is a contradiction.
\end{proof}

As noted after Definition \ref{def.minimal.FPS}, a minimal FPS minimally imposes every cycle in its precedence digraph. Hence, we obtain the following corollary.

\begin{corollary}\label{corollary.minimal.implies.no.redundancy}
    If $F \subseteq A_P$ is a minimal FPS, then $F$ has no redundant incoming propagations.
\end{corollary}

Then, we can prove a characterization of them.

\begin{proposition}\label{prop.minimally.impose}
    Let $F \subseteq A_P$ be an FPS and let $C = (e_1,\dots, e_r) \in \CC(D_F)$.
    For any $\tilde F \subseteq F$, $\tilde F$ is an FPS that minimally imposes $C$ if and only if $\tilde F = \{p_1, \ldots, p_r\}$ where $p_i \in \varphi(e_i)$ for all $i \in \{1,\ldots,r\}$.
    Moreover, $|\tilde F| = |C|$.
\end{proposition}

\begin{proof}
Let $\tilde F \subseteq F$ be an FPS that minimally imposes $C$.
By Lemma \ref{lemma.minimally.impose.implies.no.redundancy}, $\tilde F$ has no redundant incoming propagations.
Since $\tilde F$ minimally imposes $C$, we have $C \in \CC(D_{\tilde F})$.
By Remark \ref{remark.preliminary.results}, for every $e \in C$, $|\varphi(e) \cap \tilde F| = 1$, and this unique propagation can be denoted by $\varphi_{\tilde F}(e)$.

Let $H = \{\varphi_{\tilde F}(e_i) : e_i \in C\}$, which clearly satisfies $H \subseteq \tilde F$.
In addition, by Remark \ref{remark.preliminary.results}, for all distinct arcs $e_i, e_j \in C$, $\varphi(e_i) \cap \varphi(e_j) = \emptyset$. Then, $|H| = |C|$.
Next, we show that $\tilde F \subseteq H$.
Let $p \in \tilde F$.
By Remark \ref{remark.FPS.minimally.impose.cycle}, there exists $e \in C$ such that $\varphi(e) \cap \tilde F = \{p\}$.
Hence, $p = \varphi_{\tilde F}(e)$, implying that $p \in H$. 

For the converse, let $\tilde F \subseteq F$ such that $\tilde F = \{p_1, \ldots, p_r\}$, where $p_i \in \varphi(e_i)$ for all $i \in \{1,\ldots,r\}$.
It is immediate that $C \in \CC(D_{\tilde F})$.

Let $p_i \in \tilde F$ and suppose that $C \in \CC(D_{\tilde F \setminus \{p_i\}})$. 
Then $e_i$ is imposed by some propagation $p_j \in \tilde F \setminus \{p_i\}$.
Hence, $p_j \in \varphi(e_i) \cap \varphi(e_j)$, and by Remark \ref{remark.preliminary.results}, $e_i$ and $e_j$ share the same target vertex.
Consequently, this vertex has two incoming arcs in $C$, contradicting the fact that $C$ is a cycle.
Therefore, $C \notin \CC(D_{\tilde F \setminus \{p_i\}})$ for all $p_i \in \tilde F$, completing the proof that $\tilde F$ minimally imposes $C$.
\end{proof}

We illustrate the characterization given in Proposition \ref{prop.minimally.impose} with the example in Figure \ref{fig.efps}.
Consider the set $F = \{(a,b),(c,b),(c,d),(a,d)\}$ and the cycle $C = ((b,d),(d,b))$ of $D_F$.
Observe that $\varphi(b,d) = \{(a,d),(c,d)\}$ and $\varphi(d,b) = \{(a,b),(c,b)\}$.
The sets $\tilde F_1 = \{(c,b),(a,d)\}$ and $\tilde F_2 = \{(a,b),(c,d)\}$, illustrated in Figure \ref{fig.efps.left} and Figure \ref{fig.efps.right}, respectively, are 2 of the 4 FPSs that minimally impose $C$.
Indeed, the total number of such structures is determined by the following corollary, which follows from the fact that, for each $e \in C$, we choose exactly one propagation from $\varphi(e) \cap F$.

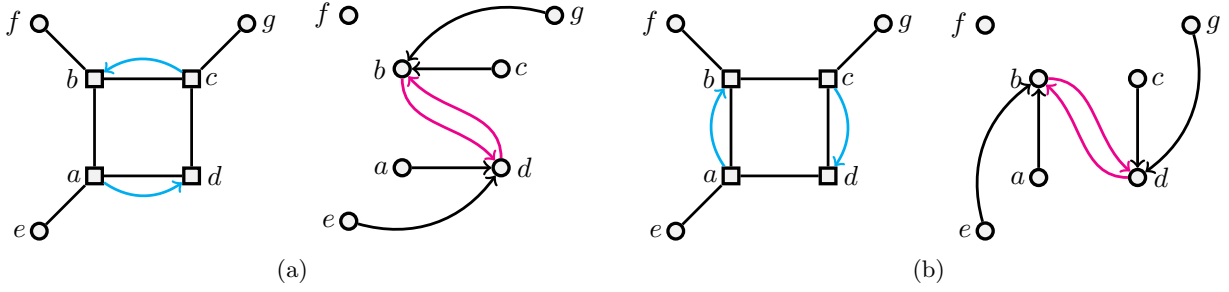
\begin{figure}

\begin{subfigure}{0.49\textwidth}
\centering
\resizebox{0.49\linewidth}{!}{
\begin{tikzpicture}[
every node/.style={circle, draw=black, fill=black!5, very thick, minimum size=2mm, inner sep=0},
every path/.style={line width=1pt},
scale=0.8
]
\node[label=left:$a$, rectangle] (a) {};
\node[label=left:$b$, above=of a, rectangle] (b) {};
\node[label=right:$c$, right=of b, rectangle] (c) {};
\node[label=right:$d$, right=of a, rectangle] (d) {};
\node[label=left:$e$, below left= 0.5cm and 0.5cm of a] (e) {};
\node[label=left:$f$, above left= 0.5cm and 0.5cm of b] (f) {};
\node[label=right:$g$, above right= 0.5cm and 0.5cm of c] (g) {};
\draw (a) -- (b) -- (c) -- (d) -- (a);
\draw (a) -- (e);
\draw (c) -- (g);
\draw (b) -- (f);
\draw[ ->, Cyan] (a) to [bend right=35] (d);
\draw[ ->, Cyan] (c) to [bend right=35] (b);
\end{tikzpicture}
}
\resizebox{0.49\linewidth}{!}{
\begin{tikzpicture}[
every node/.style={circle, draw=black, fill=black!5, very thick, minimum size=2mm, inner sep=0},
every path/.style={line width=1pt},
scale=0.8
]
\node[label=left:$a$] (a) {};
\node[label=left:$b$, above=of a] (b) {};
\node[label=right:$c$, right=of b] (c) {};
\node[label=right:$d$, right=of a] (d) {};
\node[label=left:$e$, below left= 0.5cm and 0.5cm of a] (e) {};
\node[label=left:$f$, above left= 0.5cm and 0.5cm of b] (f) {};
\node[label=right:$g$, above right= 0.5cm and 0.5cm of c] (g) {};
\draw[ ->] (a) to (d);
\draw[<-, Magenta] (b) to [out=310, in=90] (d);
\draw[ ->] (e) to [bend right=35] (d);
\draw[ ->] (c) to (b);
\draw[<-, Magenta] (d) to [out=130, in=270] (b);
\draw[ ->] (g) to [bend right=35] (b);
\end{tikzpicture}
}
\caption{}
\label{fig.efps.left}
\end{subfigure}
\hfill
\begin{subfigure}{0.49\textwidth}
\centering
\resizebox{0.49\linewidth}{!}{
\begin{tikzpicture}[
every node/.style={circle, draw=black, fill=black!5, very thick, minimum size=2mm, inner sep=0},
every path/.style={line width=1pt},
scale=0.8
]
\node[label=left:$a$, rectangle] (a) {};
\node[label=left:$b$, above=of a, rectangle] (b) {};
\node[label=right:$c$, right=of b, rectangle] (c) {};
\node[label=right:$d$, right=of a, rectangle] (d) {};
\node[label=left:$e$, below left= 0.5cm and 0.5cm of a] (e) {};
\node[label=left:$f$, above left= 0.5cm and 0.5cm of b] (f) {};
\node[label=right:$g$, above right= 0.5cm and 0.5cm of c] (g) {};
\draw (a) -- (b) -- (c) -- (d) -- (a);
\draw (a) -- (e);
\draw (c) -- (g);
\draw (b) -- (f);
\draw[ ->, Cyan] (a) to [bend left=35] (b);
\draw[ ->, Cyan] (c) to [bend left=35] (d);
\end{tikzpicture}
}
\resizebox{0.49\linewidth}{!}{
\begin{tikzpicture}[
every node/.style={circle, draw=black, fill=black!5, very thick, minimum size=2mm, inner sep=0},
every path/.style={line width=1pt},
scale=0.8
]
\node[label=left:$a$] (a) {};
\node[label=left:$b$, above=of a] (b) {};
\node[label=right:$c$, right=of b] (c) {};
\node[label=right:$d$, right=of a] (d) {};
\node[label=left:$e$, below left= 0.5cm and 0.5cm of a] (e) {};
\node[label=left:$f$, above left= 0.5cm and 0.5cm of b] (f) {};
\node[label=right:$g$, above right= 0.5cm and 0.5cm of c] (g) {};
\draw[ ->] (a) to (b);
\draw[ ->] (e) to [bend left=35] (b);
\draw[ ->, Magenta] (d) to [out=180, in=320] (b);
\draw[ ->] (c) to (d);
\draw[ ->, Magenta] (b) to [out=0, in=140] (d);
\draw[ ->] (g) to [bend left=35] (d);
\end{tikzpicture}
}
\caption{}
\label{fig.efps.right}
\end{subfigure}

\caption{Two pairs of propagation sets (in cyan) and their corresponding precedence digraphs, where each propagation set is a minimal FPS imposing the magenta cycle.}
\label{fig.efps}
\end{figure}

\begin{corollary}
    Let $F \subseteq A_P$ be an FPS and let $C = (e_1,\dots, e_r) \in \CC(D_F)$.
    The number of subsets of $F$ that are FPSs that minimally impose $C$ is $\prod_{i = 1}^r |\varphi(e_i) \cap F|$.
\end{corollary}

Recall that an FPS that minimally imposes a cycle need not be a minimal FPS.
The next lemma shows that this property holds when the cycle is chordless.

\begin{lemma}\label{lemma.chordless.implies.minimal}
    Let $F \subseteq A_P$ be an FPS. For every chordless cycle $C \in \CC(D_F)$, any FPS $\tilde F \subseteq F$ that minimally imposes $C$, is a minimal FPS.
\end{lemma}

\begin{proof}
Let $C = (e_1, e_2, \dots, e_r)$ be a  chordless cycle of $D_F$ and let $\tilde F \subseteq F$ be an FPS that minimally imposes $C$.
By Proposition \ref{prop.minimally.impose}, we have $\tilde F = \{p_1,\ldots,p_r\}$, where $p_i \in \varphi(e_i)$ for all $i \in \{1,\ldots,r\}$, and $|\tilde F| = |C|$.

Suppose that $\tilde F$ is not a minimal FPS.
Hence, there exists a proper subset $\tilde F'$ of $\tilde F$ such that $\tilde F'$ is a minimal FPS.
By Definition \ref{def.FPS}, the precedence digraph $D_{\tilde F'}$ contains a cycle $\tilde C$.
We show that the existence of $\tilde C$ implies that $C$ has a chord, contradicting the assumption that $C$ is chordless.

Since $\tilde F'$ is a minimal FPS, $\tilde F'$ minimally imposes $\tilde C$.
By Lemma \ref{lemma.minimally.impose.implies.no.redundancy}, $\tilde F'$ has no redundant incoming propagations.
Moreover, it follows from Proposition \ref{prop.minimally.impose} that $|\tilde F'| = |\tilde C|$. Hence $|\tilde C| < |C|$.

We claim that every vertex of $\tilde C$ is also a vertex of $C$.
Let $v$ be a vertex of $\tilde C$.
Since $\tilde C$ is a  cycle, it contains an incoming arc to $v$, say $(w, v) \in \tilde C$.
By Remark \ref{remark.preliminary.results}, $|\varphi(w, v) \cap \tilde F'| = 1$, and let $\varphi_{\tilde F'}(w,v)$ be this unique element.
Because $\tilde F' \subseteq \tilde F$, $\varphi_{\tilde F'}(w, v) = p_k$ for some $p_k \in \tilde F$.
By Remark \ref{remark.preliminary.results}, $p_k \in \varphi(e_k)$ and $\varphi_{\tilde F'}(w, v) = p_k$ imply that $e_k$ has target $v$.
Thus, since $e_k \in C$, $v$ is a vertex of $C$.

Since all vertices of $\tilde C$ belong to $C$ and $|\tilde C| < |C|$, $\tilde C$ contains at least one arc not belonging to $C$. Such an arc is a chord of $C$. 
\end{proof}

In order to establish the converse of Lemma \ref{lemma.chordless.implies.minimal}, we first observe that the precedence digraph associated with any minimal FPS contains exactly one cycle. 

\begin{lemma}\label{lemma.minimal.implies.exactly.one.cycle}
If $\tilde F \subseteq A_P$ is a minimal FPS, then $|\CC(D_{\tilde F})| = 1$.
\end{lemma}

\begin{proof}
Let $\tilde F \subseteq A_P$ be a minimal FPS.
By Corollary \ref{corollary.minimal.implies.no.redundancy}, $\tilde F$ has no redundant incoming propagations.
Then, by Remark \ref{remark.preliminary.results}, $|\varphi(e) \cap \tilde F| = 1$ for every arc $e$ of $D_{\tilde F}$, and we denote this unique propagation by $\varphi_{\tilde F}(e)$.
By Definition $\ref{def.FPS}$, $\CC(D_{\tilde F}) \neq \emptyset$.

Suppose, for contradiction, that $D_{\tilde F}$ contains two distinct cycles $C_1$ and $C_2$.
For each $i \in \{1,2\}$, define $\tilde F_i = \{\varphi_{\tilde F}(e) : e \in C_i\} \subseteq \tilde F$ and, by Proposition \ref{prop.minimally.impose}, $\tilde F_i$ is an FPS that minimally imposes $C_i$.
Since $\tilde F$ is a minimal FPS, we obtain $\tilde F = \tilde F_1 = \tilde F_2$.
Remark \ref{remark.preliminary.results} ensures that $\varphi$ preserves target vertices; consequently, $C_1$ and $C_2$ have the same vertex set.

As the cycles are distinct, there exists an arc $e = (u,v) \in C_1 \setminus C_2$.
Then $e$ is a chord of $C_2$ and, together with the path in $C_2$ from $v$ to $u$, forms a cycle $\tilde C_2$ in $D_{\tilde F}$ shorter than $C_2$.
Now, define $\tilde F' = \{\varphi_{\tilde F}(e): e \in \tilde C_2\} \subseteq \tilde F$ and, by Proposition \ref{prop.minimally.impose}, $\tilde F'$ is an FPS that minimally imposes $\tilde C_2$.
Observe that $|\tilde F'| = |\tilde C_2| < |C_2| = |\tilde F_2| = |\tilde F|$.
Therefore, $\tilde F'$ is an FPS and proper subset of $\tilde F$, contradicting the minimality of $\tilde F$.
\end{proof}

As a consequence of Lemma \ref{lemma.minimal.implies.exactly.one.cycle}, the unique cycle of $D_{\tilde F}$ is chordless; however, this property need not hold when the cycle is viewed in the precedence digraph $D_F$ associated with an FPS $F$ containing $\tilde F$.
This is illustrated in Figure \ref{fig.chord}.
Let $F = \{(a,b), (c,d), (e,f), (a,f)\}$.
In Figure \ref{fig.chord.left}, the subset $\tilde F = \{(a,b), (c,d), (e,f)\}$ of $F$ is a minimal FPS whose precedence digraph $D_{\tilde F}$ has a unique and chordless cycle.
In contrast, this cycle has a chord in $D_F$, as shown in Figure \ref{fig.chord.right}.

\begin{figure}

\begin{subfigure}{0.49\textwidth}
\resizebox{\linewidth}{!}{
\begin{tikzpicture}[
every node/.style={circle, draw=black, fill=black!5, very thick, minimum size=2mm, inner sep=0},
every path/.style={line width=1pt},
scale=0.8
]
\foreach \x/\s in {0/rectangle,1/circle,2/rectangle,3/circle,4/rectangle,5/circle}{
    \node[\s] (\x) at (60-60*\x:1.5) {};
}
\foreach \x/\v in {0/a,1/b,2/c,3/d,4/e,5/f}{
    \node[draw=none, fill=none] (\v) at (60-60*\x:1.9) {$\v$};
}
\foreach \x/\v in {0,1,2,3,4,5}{
    \pgfmathtruncatemacro{\y}{mod(\x+1,6)}
    \draw (\x) -- (\y);
}
\draw[->, Cyan] (0) to [bend left] (1);
\draw[->, Cyan] (2) to [bend left] (3);
\draw[->, Cyan] (4) to [bend left] (5);
\end{tikzpicture}
\begin{tikzpicture}[
every node/.style={circle, draw=black, fill=black!5, very thick, minimum size=2mm, inner sep=0},
every path/.style={line width=1pt},
scale=0.8
]
\foreach \x in {0,1,2,3,4,5}{
    \node (\x) at (60-60*\x:1.5) {};
}
\foreach \x/\v in {0/a,1/b,2/c,3/d,4/e,5/f}{
    \node[draw=none, fill=none] (\v) at (60-60*\x:1.9) {$\v$};
}
\draw[->, bend left, ] (0) to (1);
\draw[->, bend left=15, Magenta] (5) to (1);
\draw[->, bend left] (2) to (3);
\draw[->, bend left=15, Magenta] (1) to (3);
\draw[->, bend left] (4) to (5);
\draw[->, bend left=15, Magenta] (3) to (5);
\end{tikzpicture}
}
\caption{}
\label{fig.chord.left}
\end{subfigure}
\hfill
\begin{subfigure}{0.49\textwidth}
\resizebox{\linewidth}{!}{
\begin{tikzpicture}[
every node/.style={circle, draw=black, fill=black!5, very thick, minimum size=2mm, inner sep=0},
every path/.style={line width=1pt},
scale=0.8
]
\foreach \x/\s in {0/rectangle,1/circle,2/rectangle,3/circle,4/rectangle,5/circle}{
    \node[\s] (\x) at (60-60*\x:1.5) {};
}
\foreach \x/\v in {0/a,1/b,2/c,3/d,4/e,5/f}{
    \node[draw=none, fill=none] (\v) at (60-60*\x:1.9) {$\v$};
}
\foreach \x/\v in {0,1,2,3,4,5}{
    \pgfmathtruncatemacro{\y}{mod(\x+1,6)}
    \draw (\x) -- (\y);
}
\draw[->, Cyan] (0) to [bend left] (1);
\draw[->, Cyan] (2) to [bend left] (3);
\draw[->, Cyan] (4) to [bend left] (5);
\draw[->, Cyan, dashed] (0) to [bend right] (5);
\end{tikzpicture}
\begin{tikzpicture}[
every node/.style={circle, draw=black, fill=black!5, very thick, minimum size=2mm, inner sep=0},
every path/.style={line width=1pt},
scale=0.8
]
\foreach \x in {0,1,2,3,4,5}{
    \node (\x) at (60-60*\x:1.5) {};
}
\foreach \x/\v in {0/a,1/b,2/c,3/d,4/e,5/f}{
    \node[draw=none, fill=none] (\v) at (60-60*\x:1.9) {$\v$};
}
\draw[->, bend left, ] (0) to (1);
\draw[->, bend left=15, Magenta] (5) to (1);
\draw[->, bend left] (2) to (3);
\draw[->, bend left=15, Magenta] (1) to (3);
\draw[->, bend left] (4) to (5);
\draw[->, bend left=15, Magenta] (3) to (5);
\draw[->, bend right, dashed] (0) to (5);
\draw[->, bend left=15, dashed, Magenta] (1) to (5);
\end{tikzpicture}
}
\caption{}
\label{fig.chord.right}
\end{subfigure}

\caption{(a) The cyan FPS imposes the magenta  chordless cycle in the precedence digraph; (b) the dashed cyan redundant propagation adds the dashed magenta chord.}
\label{fig.chord}
\end{figure}
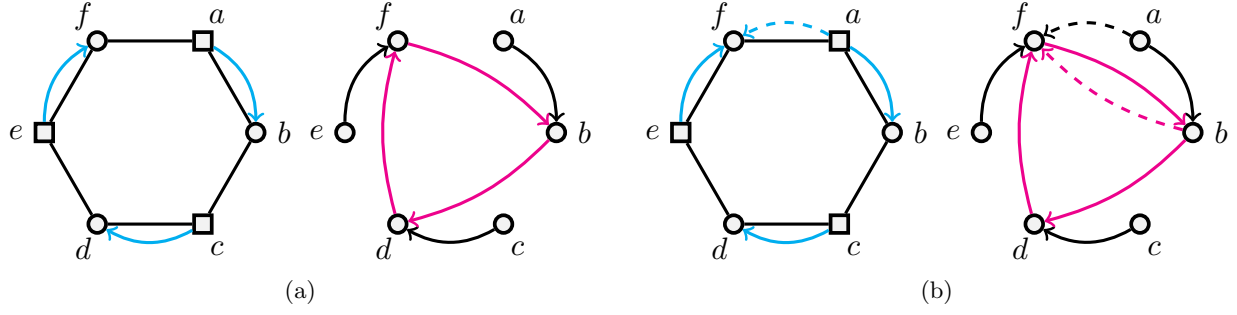

In the absence of redundant incoming propagations, minimal FPSs contained in $F$ define chordless cycles in $D_F$.

\begin{lemma}\label{lemma.minimal.implies.chordless}
    Let $F \subseteq A_P$ be an FPS without redundant incoming propagations. 
    For every minimal FPS $\tilde F \subseteq F$, the unique cycle $C_{\tilde F}$ of $D_{\tilde F}$ is chordless in $D_F$.
\end{lemma}

\begin{proof}
Let $\tilde F \subseteq F$ be a minimal FPS.
By Lemma \ref{lemma.minimal.implies.exactly.one.cycle}, $D_{\tilde F}$ contains a unique cycle $C_{\tilde F}$.
Since $\tilde F \subseteq F$, $D_{\tilde F}$ is a subgraph of $D_F$, and thus $C_{\tilde F}$ is also a cycle of $D_F$.

We first characterize the structure of $\tilde F$.
Since $F$ has no redundant incoming propagations, Remark \ref{remark.preliminary.results} ensures that $|\varphi(e) \cap F| = 1$ for every arc $e \in D_F$, and this unique propagation is denoted by $\varphi_F(e)$.
Define $\tilde F' = \{\varphi_F(e) : e \in C_{\tilde F}\} \subseteq F$ and, by Proposition \ref{prop.minimally.impose}, $\tilde F'$ is an FPS that minimally imposes $C_{\tilde F}$. 
We claim that $\tilde F' = \tilde F$.
Let $e \in C_{\tilde F}$. 
Since $C_{\tilde F}$ is a cycle of $D_{\tilde F}$, by Remark \ref{remark.preliminary.results} we have $\varphi(e) \cap \tilde F \neq \emptyset$.
As $\tilde F \subseteq F$, it follows that $\varphi(e) \cap \tilde F \subseteq \varphi(e) \cap F = \{\varphi_F(e)\}$, and hence $\varphi_F(e) \in \tilde F$. Thus, $\tilde F' \subseteq \tilde F$ and $\tilde F'$ is an FPS.
Finally, since $\tilde F$ is a minimal FPS, we conclude that $\tilde F = \tilde F' = \{\varphi_F(e) : e \in C_{\tilde F}\}$.

Suppose, for contradiction, that $C_{\tilde F}$ has a chord $(c_1,c_2)$ in $D_F$.
Then, $D_{F}$ contains a cycle $C'$ shorter than $C_{\tilde F}$, consisting of $(c_1,c_2)$ together with the path in $C_{\tilde F}$ from $c_2$ to $c_1$.
Now, define $F' = \{\varphi_{F}(e): e \in C'\} \subseteq F$ and, by Proposition \ref{prop.minimally.impose}, $F'$ is an FPS that minimally imposes $C'$.
We prove that $F' \subseteq \tilde F$.
For every arc $e \in C' \setminus \{(c_1,c_2)\}$, it holds that $\varphi_F(e) \in \tilde F$, since $e$ is also an arc of $C_{\tilde F}$.
For the remaining arc $(c_1,c_2)$, note that $c_2$ is vertex of $C_{\tilde F}$, so there exists an arc $(u,c_2)$ in $C_{\tilde F}$.
Since $F$ has no redundant incoming propagations and $(c_1,c_2)$ and $(u,c_2)$ have the same target, Remark \ref{remark.preliminary.results} ensures $ \varphi_F(c_1,c_2) = \varphi_F(u,c_2) \in \tilde F$.

Observe that $|F'| = |C'| < |C_{\tilde F}| = |\tilde F|$. Therefore, $F'$ is an FPS and proper subset of $\tilde F$, contradicting the minimality of $\tilde F$.
\end{proof}

Lemmas \ref{lemma.chordless.implies.minimal} and \ref{lemma.minimal.implies.chordless} imply the following complete characterization.

\begin{proposition}\label{prop.characterization.minimal.FPS}
    Let $F \subseteq A_P$ be an FPS. 
    \begin{enumerate}
    \item For every  chordless cycle $C \in \CC(D_F)$, any FPS $\tilde F \subseteq F$ that minimally imposes $C$ is a minimal FPS.
    \label{prop.characterization.minimal.FPS.part.1}
    \item If $F$ has no redundant incoming propagations, for every minimal FPS $\tilde F \subseteq F$, the unique cycle $C_{\tilde F}$ of $D_{\tilde F}$ is chordless in $D_F$.
    \label{prop.characterization.minimal.FPS.part.2} 
    \end{enumerate}
\end{proposition}

\subsection{Forbidden propagation sets of minimum cardinality}

Given a graph $G=(V,E)$ and $V_P \subseteq V$, we now investigate the structure of the family $\mathcal{F}_2(G,V_P)$ of minimal FPSs $F$ such that $D_F$ contains a 2-cycle.
This family corresponds to the FPSs of minimum cardinality.

Figure \ref{fig.fps.length.2} illustrates FPSs whose membership in $\mathcal{F}_2(G, V_P)$ is straightforward.
In this figure, rectangular vertices are required to belong to $V_P$, whereas circular vertices may or may not belong to $V_P$, and dashed edges may or may not belong to $G$. 
In addition, $\mathcal{F}_2(G,V_P)$ is completely characterized by these four cases.
Before presenting the proof, we make the following observation.
Let $e = (w,v)$ be an arc of $D$ and let $p \in \varphi(e)$. 
We say that $e$ is a \emph{direct} precedence for $p$ if $p = e$, in which case $w \in V_P$ and the edge $wv$ belongs to $G$. 
Otherwise, $e$ is an \emph{indirect} precedence for $p$, meaning that $p = (u,v)$ for some $u \in V_P$ such that $u$ is a common neighbor of $w$ and $v$, hence both edges $uv$ and $uw$ belong to $G$.

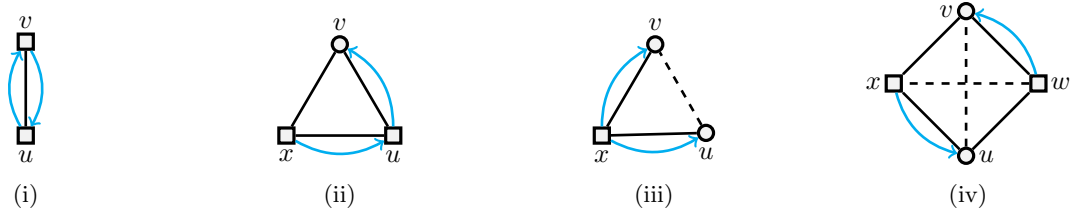
\begin{figure}

\captionsetup[subfigure]{labelformat=empty}

\begin{subfigure}[t]{0.24\textwidth}
\centering
\begin{tikzpicture}[
every node/.style={circle, draw=black, fill=black!5, very thick, minimum size=2mm, inner sep=0},
every path/.style={line width=1pt}
]
\node [label=above:$v$, rectangle] (v) {};
\node [label=below:$u$, below=of v, rectangle] (u) {};
\draw (u) -- (v);
\draw[->,cyan] (u) to [bend left] (v);
\draw[->,cyan] (v) to [bend left] (u);
\end{tikzpicture}
\caption{(i)}
\end{subfigure}
\begin{subfigure}[t]{0.25\textwidth}
\centering
\begin{tikzpicture}[
every node/.style={circle, draw=black, fill=black!5, very thick, minimum size=2mm, inner sep=0},
every path/.style={line width=1pt}
]
\node [label=below:$x$, rectangle] (x) {};
\node [label=above:$v$, above right=1cm and 0.5cm of x] (v) {};
\node [label=below:$u$, below right=1cm and 0.5cm of v, rectangle] (u) {};
\draw (x) -- (u) -- (v) -- (x);
\draw[->,cyan] (x) to [bend right] (u);
\draw[->,cyan] (u) to [bend right] (v);
\end{tikzpicture}
\caption{(ii)}
\end{subfigure}
\begin{subfigure}[t]{0.24\textwidth}
\centering
\begin{tikzpicture}[
every node/.style={circle, draw=black, fill=black!5, very thick, minimum size=2mm, inner sep=0},
every path/.style={line width=1pt}
]
\node [label=below:$x$, rectangle] (x) {};
\node [label=above:$v$, above right=1cm and 0.5cm of x] (v) {};
\node [label=below:$u$, below right=1cm and 0.5cm of v] (u) {};
\draw (x) -- (u);
\draw (x) -- (v);
\draw[dashed] (u) -- (v);
\draw[->,cyan] (x) to [bend right] (u);
\draw[->,cyan] (x) to [bend left] (v);
\end{tikzpicture}
\caption{(iii)}
\end{subfigure}
\begin{subfigure}[t]{0.25\textwidth}
\centering
\begin{tikzpicture}[
every node/.style={circle, draw=black, fill=black!5, very thick, minimum size=2mm, inner sep=0},
every path/.style={line width=1pt}
]
\node [label=left:$x$, rectangle] (x) {};
\node [label=right:$u$, below right=0.75cm and 0.75cm of x] (u) {};
\node [label=left:$v$, above right=0.75cm and 0.75cm of x] (v) {};
\node [label=right:$w$, below right=0.75cm and 0.75cm of v, rectangle] (w) {};
\draw (u) -- (x) -- (v) -- (w) -- (u);
\draw[dashed] (u) -- (v);
\draw[dashed] (w) -- (x);
\draw[->,cyan] (x) to [bend right] (u);
\draw[->,cyan] (w) to [bend right] (v);
\end{tikzpicture}
\caption{(iv)}
\end{subfigure}

\caption{
The cyan propagations are FPSs that minimally impose 2-cycles.}
\label{fig.fps.length.2}
\end{figure}

\begin{proposition}\label{prop.fps.2-cycles}
Let $G=(V,E)$ be a graph and $V_P \subseteq V$.
Then,
\begin{enumerate}
    \item For $u,v \in V_P$ such that $uv \in E$, $F = \{(u,v),(v,u)\} \in \FF_2(G,V_P)$.
    
    \item For $x,u \in V_P$ such that $xu \in E$ and $v \in N(x) \cap N(u)$, $F = \{(x,u), (u,v)\} \in \FF_2(G,V_P)$.
    
    \item For $x \in V_P$ and $u,v \in N(x)$ with $u \neq v$, $F = \{(x,u), (x,v)\} \in \FF_2(G,V_P)$.
    
    \item For $x,w \in V_P$ and $u,v \in N(u) \cap N(v)$ with $x \neq w$ and $u \neq v$, $F = \{(x,u),(w,v)\}  \in \FF_2(G,V_P)$.
\end{enumerate}
Moreover, every FPS in $\FF_2(G,V_P)$ falls into one of the above cases.
\end{proposition}

\begin{proof}


The first part follows directly, so we prove the converse. 
Let $F$ be a minimal FPS such that $D_F$ contains a 2-cycle, say $C = ((u,v),(v,u)) \in \CC(D_F)$. We show that $F$ belongs to one of the cases described above.

Since $F$ is a minimal FPS, $F$ minimally imposes $C$.
Then, by Proposition \ref{prop.minimally.impose}, $F$ consists of exactly two propagations, one in $\varphi(u,v)$, say $p_1=(w,v)$, and one in $\varphi(v,u)$, say $p_2=(x,u)$.

If both precedences of $C$ are direct for $p_1$ and $p_2$, respectively, then $w = u$, $x = v$, and $u$ and $v$ are adjacent, which yields the first case.

Otherwise, suppose that exactly one precedence of $C$ is indirect. Without loss of generality, suppose that $(v,u)$ is an indirect precedence for $p_2$, while $(u,v)$ is a direct precedence for $p_1$.
Then $w = u$, $u$ and $v$ are adjacent, and $x$ is a common neighbor of $u$ and $v$, which matches the second case.

Finally, suppose that both precedences of $C$ are indirect.
Then $w$ and $x$ are common neighbors of $u$ and $v$.
If $w = x$, we obtain the third case.
Otherwise, when $w \neq x$, we obtain the fourth case.

This exhaustively covers all possibilities, completing the description of $\FF_2(G,V_P)$.
\end{proof}

\subsection{Extended forbidden propagation sets}\label{subsec.efps}

Rather than focusing on minimal FPSs, we now adopt the opposite perspective and consider unions of FPSs sharing specific structural properties

\begin{definition}\label{def.efps}
    Given $C \in \CC(D)$, the \emph{extended forbidden propagation set} (EFPS) defined by $C$ is the union of all FPSs in $A_P$ that minimally impose $C$.
\end{definition}

\begin{remark}
    By Proposition \ref{prop.minimally.impose}, the EFPS defined by $C$ is $\bigcup_{e \in C} \varphi(e)$, which we denote by $\varphi(C)$.
    Additionally, $|\varphi(C)| = \sum_{e \in C} |\varphi(e)| \geq |C|$.
\end{remark}

For example, in Figure \ref{fig.efps}, consider the  cycle $C = ((b,d),(d,b))$, then $\varphi(C) = \varphi(b,d) \cup \varphi(d,b) = \{(a,d),(c,d)\} \cup \{(a,b),(c,b)\} = \{(a,d),(c,d),(a,b),(c,b)\}$.

Since every EFPS $\varphi(C)$ is also an FPS, Remark \ref{remark.fps} implies that at most $|\varphi(C)|-1$ propagations in $\varphi(C)$ can be applied in any proper calculation of $M(\rho)$. 
In fact, a stronger result holds.

\begin{proposition}\label{prop.efps}
Let $\varphi(C)$ be an EFPS. At most $|C| - 1$ propagations of $\varphi(C)$ can be applied in any proper calculation of $M(\rho)$.
\end{proposition}
\begin{proof}
Let $R \subseteq A_P$ be the set of propagations applied in a proper calculation of $M(\rho)$.
Suppose, for a contradiction, that $|R \cap \varphi(C)| \geq |C|$.
If $|R \cap \varphi(e)| \geq 1$ for every $e \in C$, then $R$ contains an FPS that minimally imposes $C$, which contradicts Lemma \ref{lemma.DAG}.
Otherwise, there exists an arc $e = (w,v) \in C$ such that $|R \cap \varphi(e)| \geq 2$.
Since all propagations in $\varphi(e)$ have the target $v$, this would imply that $v$ receives more than one incoming propagation from $R$, contradicting the assumption that the calculation of $M(\rho)$ is proper.
\end{proof}

The next remark highlights that minimal FPSs are always contained in EFPSs, a result that will be crucial for extending our models to this broader class of structures.

\begin{remark}\label{remark.fps.replaced.by.efps}
Let $F \subseteq A_P$ be a minimal FPS.
By Lemma \ref{lemma.minimal.implies.exactly.one.cycle}, $D_F$ contains a unique cycle denoted by $C_{F}$.
Then, $F$ minimally imposes $C_{F}$ and, by Remark \ref{remark.FPS.minimally.impose.cycle}, it follows that $F \subseteq \varphi(C_F)$.
\end{remark}

\section{Integer linear programming formulations}\label{sec.fpsip}

This section aims to develop ILP formulations for the CPDS based on FPSs.
We consider decision variables to describe a $k$-capacitated power dominating function $\rho$.
For each $v \in V$, the variable $s_v = 1$ if and only if $v \in S_{\rho}$, for each $u \in V$ and $v \in N(u)$, the variable $w_{uv} = 1$ if $v \in \rho(u)$, and for each $(u,v) \in A_P$, the variable $y_{uv} = 1$ if rule (PR) is applied at $u$ to monitor $v$ in the calculation of $M(\rho)$.
Observe that $w_{uv}$ is irrelevant when $\delta(u) \leq k$, since in that case $\rho(u)$ may contain all neighbors of $u$. 
Consequently, we introduce a variable $w_{uv}$ only for each $(u,v)$ in $A_D = \{(u,v): u \in V,\ \delta(u) > k,\ v \in N(u) \}$.
Let $\FF(G,V_P)$ be the set of minimal FPSs of $(G,V_P)$.
\begin{align}
\text{(\FPSIP)}\ \min\  & \sum_{v \in V} s_v \label{FPSIP.ObjFun}\\
s.t.\  & s_v + \sum_{\mathclap{\substack{u \in N(v):\\\delta(u) \leq k}}} s_u + \sum_{\mathclap{\substack{u \in N(v):\\\delta(u) > k}}} w_{uv} + \sum_{\mathclap{u \in N(v) \cap V_P}} y_{uv} \geq 1 & v \in V \label{FPSIP.R1}\\
& \sum_{u \in N(v)} w_{vu} \leq ks_v & v \in V: \delta(v) > k \label{FPSIP.R2}\\
& \sum_{(u,v) \in F} y_{uv} \leq |F|-1 & F \in \FF(G,V_P) \label{FPSIP.R3}\\
& s \in \{0,1\}^V,\ w \in \{0,1\}^{A_D},\ y \in \{0,1\}^{A_P}. \notag
\end{align}

Objective function \eqref{FPSIP.ObjFun} minimizes $|S_{\rho}|$. 
Constraints \eqref{FPSIP.R1} enforce $M(\rho) = V$, i.e., each vertex $v$ is monitored either because $v \in S_{\rho}$ or some neighbor $u$ monitors $v$ through rule (DR) or (PR).
In the particular case where rule (DR) is applied, if $\delta(u) \leq k$, then $u$ monitors $v$ whenever $u \in S_{\rho}$ ($s_u = 1$); but if $\delta(u) > k$, the monitoring additionally requires that $v \in \rho(u)$ ($w_{uv} = 1$).
Constraints \eqref{FPSIP.R2} ensure that $\rho$ is a $k$-capacitated function, i.e, that each vertex $v$ monitors at most $k$ neighbors through rule (DR) when $v \in S_{\rho}$ ($s_v = 1$), and none when $v \notin S_{\rho}$ ($s_v = 0$).
Constraints \eqref{FPSIP.R3}, referred to as \emph{FPS constraints}, ensure that not every propagation in an FPS is applied in the solution (see Remark \ref{remark.fps}).
The number of FPS constraints may be exponential, since FPSs are subsets of propagations.

Given $\tilde \FF \subseteq \FF(G,V_P)$, we denote by \FPSIPF{$\tilde \FF$} the formulation obtained from \FPSIP by omitting all constraints associated with minimal FPSs that do not belong to $\tilde \FF$. 
Before proving the correctness of formulation \FPSIP, we first prove that any solution of \FPSIPF{$\tilde \FF$} that does not correspond to a valid $k$-capacitated power dominating function necessarily activates a set of propagations whose precedence digraph contains a  cycle.
This result, together with its proof, will also be the basis for the development of a solution approach that dynamically generates FPS constraints in Section \ref{sec.constraint.generation}.

\begin{lemma}\label{lemma.fps.constraint.generation}
Let $\tilde \FF \subseteq \FF(G,V_P)$ and $(\tilde s, \tilde w, \tilde y)$ be a feasible solution of \FPSIPF{$\tilde \FF$}.
Define $\rho: S_{\rho} \to \PP(V)$ as $S_{\rho} = \{v \in V: \tilde s_v = 1\}$ and, for each $u \in S_{\rho}$, $\rho(u) = N(u)$ if $\delta(u) \leq k$ and $\rho(u) = \{v \in N(u): \tilde w_{uv} = 1\}$ if $\delta(u) > k$.

If $\rho$ is not a $k$-capacitated power dominating function, then the set $F = \{(u,v) \in A_P: \tilde y_{uv} = 1\}$ is an FPS.
Moreover, the precedence digraph $D_F$ contains a  cycle whose vertices all belong to $V \setminus M(\rho)$. 
\end{lemma}
\begin{proof}
First, note that $|\rho(u)| \leq k$ for all $u \in S_{\rho}$; which is immediate when $\delta(u) \leq k$, and ensured by constraints \eqref{FPSIP.R2} when $\delta(u) > k$. 
Hence, if $\rho$ is not a $k$--capacitated power dominating function, the set $N = V \setminus M(\rho)$ of unmonitored vertices is nonempty.

We show that for any $v \in N$, there exists another vertex $v' \in N$ such that $(v',v)$ is an arc of $D_F$.
Since $N$ is finite, iterating this argument yields a repeated vertex in $N$, and therefore a  cycle in $D_F$.

Fix any $v \in N$.
As $v$ is not monitored, we have $\tilde s_v = 0$, $\tilde s_u = 0$ for all $u \in N(v)$ with $\delta(u) \leq k$, and $\tilde 
w_{uv} = 0$ for all $u \in N(v)$ with $\delta(u) > k$; otherwise, rule (DR) could be applied at $v$ to monitor itself, or at $u$ to monitor $v$.
Thus, for constraint \eqref{FPSIP.R1} to hold at $v$, there must exist at least one neighbor $u \in V_P$ such that $\tilde y_{uv} = 1$, and consequently $(u,v) \in F$.
Since $v$ is not monitored, we have $u \in N$ or $u$ has another neighbor $w \in N \setminus \{v\}$; otherwise, rule (PR) could be applied at $u$ to monitor $v$.
Therefore, by defining $v' = u$ when $u \in N$ and $v' = w$ otherwise, we obtain that $(v',v) \in \psi(u,v)$ with $v' \in N$.
Since $(u,v) \in F$, the precedence $(v',v)$ is an arc of $D_F$, completing the proof.
\end{proof}

We can therefore conclude the correctness of the formulation.
 
\begin{proposition}\label{prop.FPS.IP}
\FPSIP is a correct formulation for CPDS.
\end{proposition}

\begin{proof}
Let $\rho$ be a $k$-capacitated power dominating function. 
Thus, there exists a finite sequence of applications of rules (DR) and (PR) such that $M(\rho) = V$, which we may assume to be proper.
We first show that there exists a feasible solution of \FPSIP that encodes $\rho$.

Consider the vector $(\tilde s, \tilde w, \tilde y) \in \{0,1\}^{V \times A_D \times A_P}$ defined as: $\tilde s_v = 1$ if and only if $v \in S_{\rho}$, for all $v \in V$; $\tilde w_{uv} = 1$ if and only if the rule (DR) is applied at $u$ to monitor $v$, for all $(u,v) \in A_D$; $\tilde y_{uv} = 1$ if and only if the rule (PR) is applied at $u$ to monitor $v$, for all $(u,v) \in A_P$.
We will see that $(\tilde s, \tilde w, \tilde y)$ verifies the constraints of \FPSIP and therefore is a feasible solution.

To prove the validity of constraints \eqref{FPSIP.R1}, let $v \in V$. 
As $M(\rho) = V$, $v$ must be monitored by rule (DR) or (PR).
If $v$ monitors itself through rule (DR), then $\tilde s_v = 1$.
But if a neighbor $u$ monitors $v$ through rule (DR), then $\tilde s_u = 1$ when $\delta(u) \leq k$, and $\tilde w_{uv} = 1$ otherwise.
Finally, if some neighbor $u \in V_P$ monitors $v$ through rule (PR), then $\tilde y_{uv} = 1$.
Thus, the left-hand side is at least 1.

For constraints \eqref{FPSIP.R2}, let $v \in V$ with $\delta(v) > k$. 
Since $\rho$ is a $k$-capacitated power dominating function, we have $|\rho(v)| \leq k$. Hence, $v$ monitors at most $k$ neighbors through rule (DR), which implies $\sum_{u \in N(v)} \tilde w_{vu} \leq k$. 
Additionally, if $v \notin S_{\rho}$, rule (DR) cannot be applied at all, implying $\tilde w_{vu} = 0$ for all $u \in N(v)$. 

For constraints \eqref{FPSIP.R3}, let $R \subseteq A_P$ be the set of propagations applied by rule (PR) in the calculation of $M(\rho)$. By Lemma \ref{lemma.DAG}, $D_R$ contains no  cycles, which guarantees that there is no FPS $F$ such that $\tilde y_{uv} = 1$ for all $(u,v) \in F$.

\medskip

Now, given a feasible solution $(\tilde s, \tilde w, \tilde y)$ of \FPSIP.
Define $S_{\rho} = \{v \in V: \tilde s_v = 1\}$ and $\rho: S_{\rho} \to \PP(V)$ by setting, for each $u \in S_{\rho}$, $\rho(u) = N(u)$ if $\delta(u) \leq k$ and $\rho(u) = \{v \in N(u): \tilde w_{uv} = 1\}$ if $\delta(u) > k$.

Suppose that $\rho$ is not a $k$--capacitated power dominating function.
By Lemma \ref{lemma.fps.constraint.generation}, the precedence digraph $D_F$, where $F = \{(u,v) \in A_P: \tilde y_{uv} = 1\}$, contains a  cycle.
Hence, $F$ is an FPS.
Therefore, $F$ contains a minimal FPS $F'$, and the FPS constraint associated with $F'$ would be violated by $(\tilde s, \tilde w, \tilde y)$, contradicting the feasibility of the solution.

\medskip

Finally, in the proposed constructions for both directions, the cardinality of $S_{\rho}$ is equal to the objective value of the solution $(\tilde s, \tilde w, \tilde y)$.
\end{proof}

In the remainder of this section, we introduce some enhancements for \FPSIP.

\subsection{Incoming and outgoing propagation constraints}\label{sec.incoming.outgoing.constraints}

We now introduce additional constraints that are not strictly required for the correctness of the formulation, but that strengthen the resulting models. 
These constraints are inspired by the work of \citet{Jovanovic2020}, and are here reinforced and adapted to consider the zero-forcing property.

The first constraints forbid redundant incoming propagations.
Although they are not strictly valid inequalities, since they may cut off feasible solutions, they only eliminate non-proper calculations of $M(\rho)$ and therefore do not affect the optimal value of \FPSIP. 
We refer to them as the \emph{incoming propagation constraints} (InP).
\begin{align}
& \sum_{u \in N(v) \cap V_P} y_{uv} \leq 1 & v \in V. \label{FPSIP.In.Prop}
\end{align}

Similarly, \citet{Jovanovic2020} propose the following constraints that forbid multiple outgoing propagations from a vertex.
\begin{align}
&\sum_{u \in N(v)} y_{vu} \leq 1 & v \in V_P.
\end{align}
When $\delta(v) = 2$, these constraints are FPS constraints, since the pair of outgoing propagations from $v$, say $F = \{(v,u),(v,w)\}$, imposes the cycle $((u,w),(w,u)) \in \CC(D_F)$ and is therefore a minimal FPS.
Otherwise, when $\delta(v) > 2$, they dominate the FPSs constraints associated with any pair of outgoing propagations from $v$.

We further strengthen these constraints as follows.
If a vertex $v \in V_P$ satisfies $\delta(v) \leq k$, then each neighbor $u$ can be monitored by applying rule (DR) instead of rule (PR).
Consequently, whenever $s_v = 1$, no outgoing propagation from $v$ is required, and we enforce $y_{vu} = 0$.
This strengthening does not necessarily hold when $\delta(v) > k$, since in that case $v$ may need to monitor some of its neighbors using rule (PR).
We refer to these constraints as the \emph{outgoing propagation constraints} (OutP).
\begin{align}
&\sum_{u \in N(v)} y_{vu} \leq 1 - s_v & v \in V_P: \delta(v) \leq k \label{contraints.OutP.first} \\ 
&\sum_{u \in N(v)} y_{vu} \leq 1 & v \in V_P: \delta(v) > k.
\end{align}
As with constraints (InP), constraints (\ref{contraints.OutP.first}) are not strictly valid inequalities but only eliminate calculations of $M(\rho)$ with unnecessary propagations.

\subsection{Formulation based on extended forbidden propagation sets}

We use EFPSs to derive an alternative and tighter formulation for CPDS.
We introduce the following constraints, called \emph{EFPS constraints}, which limit the number of propagations of any EFPS that can be applied in a solution.
\begin{align}\label{FPSIP.EFPS}
& \sum_{(u,v) \in \varphi(C)} y_{uv} \leq |C|-1 & C \in \CC(D).
\end{align}
As a consequence of Propositions \ref{prop.efps}, this family of inequalities is valid for any solution which involves a proper calculation of $M(\rho)$.

By Remark \ref{remark.fps.replaced.by.efps}, every FPS constraint is dominated by an EFPS constraint.
Then, FPS constraints can be replaced with EFPS constraints, resulting in a tighter formulation for CPDS.
For completeness, we now present this alternative formulation.
\begin{align*}
\text{(\EFPSIP)}\ \min\  & \sum_{v \in V} s_v\\
s.t.\  & s_v + \sum_{\mathclap{\substack{u \in N(v):\\\delta(u) \leq k}}} s_u + \sum_{\mathclap{\substack{u \in N(v):\\\delta(u) > k}}} w_{uv} + \sum_{\mathclap{u \in N(v) \cap V_P}} y_{uv} \geq 1 & v \in V \\
& \sum_{u \in N(v)} w_{vu} \leq ks_v & v \in V: \delta(v) > k \\
& \sum_{(u,v) \in \varphi(C)} y_{uv} \leq |C|-1 & C \in \mathscr{C}(D)\\
& s \in \{0,1\}^V,\ w \in \{0,1\}^{A_D},\ y \in \{0,1\}^{A_P}. \notag
\end{align*}

\subsection{Solving \FPSIP and \EFPSIP: dynamic generation of constraints}\label{sec.constraint.generation}

The exponential number of constraints in \FPSIP and \EFPSIP motivates a \emph{lazy-constraint} approach embedded within a generic framework for solving ILP.
For \FPSIP, starting from a (possibly empty) subset of FPS constraints associated with a subset $\tilde \FF \subseteq \FF(G, V_P)$, the algorithm iteratively solves \FPSIPF{$\tilde \FF$} and adds missing constraints only when these are violated by an obtained solution.
We present in Algorithm \ref{algo.lazy.separation} a polynomial-time separation procedure for FPS constraints, supported by the results introduced in the previous sections.

\begin{algorithm}
\caption{Separation of FPS constraints}\label{algo.lazy.separation}
\begin{algorithmic}[1]
\Require integer solution $(\tilde s, \tilde w,\tilde y)$ of \FPSIPF{$\tilde \FF$}
\Ensure minimal FPS not in $\tilde \FF$ whose associated constraint is violated (if any) 
\State $\rho \gets$ $k$-capacitated function defined by $(\tilde s, \tilde w,\tilde y)$
\State $M(\rho) \gets$ monitored set
\If{$M(\rho) = V$}
    \Return none
\EndIf
\State $C \gets$ find a  cycle in $D_{F}$, where $F = \{(u, v) \in A_P: \tilde y_{uv} = 1\}$ [Lemma \ref{lemma.fps.constraint.generation}]
\State $\tilde C = (e_1,\ldots,e_r) \gets $ chordless  cycle obtained from $C$
\State $\tilde F \gets \{p_i \in \varphi(e_i) \cap F: e_i \in \tilde C\}$ [Proposition \ref{prop.characterization.minimal.FPS}]
\State \Return $\tilde F$
\end{algorithmic}
\end{algorithm}

This algorithm begins by constructing the $k$-capacitated function defined by a feasible solution $(\tilde{s}, \tilde{w}, \tilde{y})$ of \FPSIPF{$\tilde \FF$}, according to Lemma \ref{lemma.fps.constraint.generation}. 
The monitored set $M(\rho)$ is then calculated by an iterative procedure that applies rules (DR) and (PR) whenever possible, until no additional vertices can be monitored.

If $M(\rho) = V$, then $\rho$ is a $k$-capacitated power dominating function. 
Note that this does not necessarily imply that $(\tilde{s}, \tilde{w}, \tilde{y})$ is a feasible solution of \FPSIP, as it may contain redundant or unnecessary propagations.
However, a feasible solution with the same objective value can always be obtained from a proper calculation of $M(\rho)$, following Proposition \ref{prop.FPS.IP}.

Otherwise, when $M(\rho) \neq V$, we identify a cycle $C$ in $D_F$, where $F = \{(u,v) \in A_P : \tilde{y}_{uv} = 1\}$, whose existence is guaranteed by Lemma \ref{lemma.fps.constraint.generation}.
Following the proof of this lemma, the construction of $C$ starts from an arbitrary vertex $v$ in the set $N$ of unmonitored vertices and iteratively moves backwards to a randomly selected unmonitored predecessor of $v$ in $D_F$.
If $C$ has chords, we greedily trim $C$ until a chordless cycle $\tilde C$ is obtained: whenever a chord $(u,v)$ is detected, we take the cycle defined by $(u,v)$ and the path from $v$ to $u$.

By Proposition \ref{prop.characterization.minimal.FPS}, any FPS $\tilde F \subseteq F$ that minimally imposes $\tilde C$ is a minimal FPS.
We construct $\tilde F$ by selecting, for each arc $e_i \in \tilde C$, a random propagation $p_i \in \varphi(e_i) \cap F$, see Proposition \ref{prop.minimally.impose}.
Finally, $\tilde F$ is returned.

To speed up the generation of FPS constraints, we allow generating multiple minimal FPSs in a single execution. 
To this end, we construct a cycle starting from each vertex in $N$. 
Thus, at most $|N|$ distinct cycles are generated, and a minimal FPS is derived from each cycle.

Algorithm \ref{algo.lazy.separation} can be easily adapted to generate EFPS constraints for the formulation \EFPSIP, provided that some additional information is stored.
Specifically, the precedence digraph $D$ and the mapping $\varphi$ are calculated once during a preprocessing stage.
Then, in lines 6 and 7, instead of obtaining and returning $\tilde F$, we obtain and return $\varphi(\tilde C)$.

Although the lazy-constraint approach for \FPSIP can be initialized with an empty subset of FPS constraints, preliminary experiments indicate that very short cycles are frequently identified during the early stages of the separation process. 
In Section \ref{sec.experiments}, we therefore evaluate initialization with the FPSs constraints associated with the set $\mathcal{F}_2(G,V_P)$ of minimal FPSs that impose 2-cycles (see Proposition \ref{prop.fps.2-cycles}). 
For \EFPSIP, we consider the analogous set $\mathcal{C}_2(D)$ of EFPSs defined by 2-cycles in $D$.

\section{Adapting ILP formulations from the literature}\label{sec.formulations.from.literature}

In this section, we examine the most competitive ILP formulations for solving PDS from the literature and show their adaptation to CPDS.
These adapted formulations will be used for comparison in computational experiments.

\subsection{Formulation \BRIIP}

This ILP formulation is adapted from \citet{Brimkov2019} and is denoted by \BRIIP.
Like \FPSIP, the authors consider binary variables $s_v$ for each $v \in V$ and $y_{uv}$ for each $(u,v) \in A_P$; however, their variables $y_{uv}$ indicate that $u$ monitors $v$ without specifying which rule, (DR) or (PR), is applied.
Rather than our FPS-based approach, they consider integer variables $x_v \in \{0,\ldots,T\}$ for each $v \in V$, explicitly representing the timestep at which $v$ becomes monitored, and use big-$M$ constraints to enforce the temporal precedences that arise from the application of the rules.
We propose the introduction of binary variables $w_{uv}$ for each $(u,v) \in A_D$ to handle capacity and propose the following constraints.
{\small
\begin{align}
min\  & \sum_{v \in V} s_v \label{BRIIP.ObjFun}\\ 
s.t.\ & s_v + \sum_{\mathclap{\substack{u \in N(v):\\\delta(u) \leq k}}} s_u + \sum_{\mathclap{\substack{u \in N(v):\\\delta(u) > k}}} w_{uv} + \sum_{\mathclap{u \in N(v) \cap V_P}} y_{uv} \geq 1 & v \in V \label{BRIIP.R1}\\
& \sum_{u \in N(v)} w_{vu} \leq ks_v & v \in V: \delta(v) > k \label{BRIIP.R2}\\
& x_w - x_v + (T+1)y_{uv} \leq T & (u,v) \in A_P,\ w \in N[u] \setminus \{v\} \label{BRIIP.R3}\\
& s \in \{0,1\}^V,\ w \in \{0,1\}^{A_D},\ y \in \{0,1\}^{A_P}, x \in \{0,\ldots,T\}^V. \span \notag
\end{align}
}

Constraints \eqref{BRIIP.R1} and \eqref{BRIIP.R2} are equivalent to constraints \eqref{FPSIP.R1} and \eqref{FPSIP.R2} of \FPSIP, respectively.
Constraints \eqref{BRIIP.R3} are used to enforce temporal precedences, but only for rule (PR).
In contrast to the original formulation, the right-hand side of constraints \eqref{BRIIP.R3} does not include the term $(T+1)s_u$ when $w \neq u$, since $u$ may belong to $S_{\rho}$ and still monitor $w$ through rule (PR).
For computational experiments, we set $T = |V|$, as in the computational experiments of \citet{Brimkov2019}. 

\subsection{Formulation \JOVIP}

The next ILP formulation, denoted \JOVIP, is based on \citet{Jovanovic2020}.
The authors consider the same variables as in \BRIIP, but incorporate additional inequalities.
In particular, constraints \eqref{JOVIP.R1} and \eqref{JOVIP.R2} force a vertex $v$ to be monitored at timestep 1, i.e., $x_v = 1$, when rule (DR) is applied.
Constraints \eqref{JOVIP.R3} and \eqref{JOVIP.R4} prevent more than one incoming (respectively, outgoing) propagation at the same vertex.
Constraints \eqref{JOVIP.R5} forbid simultaneous propagations in opposite directions (in analogy with FPSs, they forbid the first case of Proposition \ref{prop.fps.2-cycles}).
Constraints \eqref{JOVIP.R6} enforce temporal precedences, analogously to constraints \eqref{BRIIP.R3} in \BRIIP.
In addition, to handle capacity, we adapt constraints \eqref{JOVIP.R7} and \eqref{JOVIP.R8} directly from the original formulation, and introduce constraints \eqref{JOVIP.R9}, which are equivalent to \eqref{FPSIP.R2} in \FPSIP.
{
\small
\begin{align}
min\ & \sum_{v \in V} s_v \label{JOVIP.ObjFun}\\
s.t.\ & x_v \leq s_v + M(1-s_v) & v \in V \label{JOVIP.R1}\\
& x_v \leq s_u + M(1-s_u) & v \in V, u \in N(v) : \delta(u) \leq k \label{JOVIP.R2}\\
& \sum_{u \in N(v)} y_{vu} \leq 1 & v \in V_P \label{JOVIP.R3}\\
& \sum_{u \in N(v) \cap V_P} y_{uv} \leq 1 & v \in V \label{JOVIP.R4}\\
& y_{vu} + y_{uv} \leq 1 & (v,u),(u,v) \in A_P \label{JOVIP.R5}\\
& x_v \geq x_w + 1 - M(1-y_{uv}) & (u,v) \in A_P,\ w \in N[u] \setminus \{v\} \label{JOVIP.R6}\\
& x_v \leq w_{uv} + M(1-w_{uv}) & (u,v) \in A_D \label{JOVIP.R7}\\
& x_v \leq M(s_v + \sum_{\mathclap{\substack{u \in N(v):\\\delta(u) \leq k}}} s_u + \sum_{\mathclap{\substack{u \in N(v):\\\delta(u) > k}}} w_{uv} + \sum_{\mathclap{u \in N(v) \cap V_P}} y_{uv}) & v \in V \label{JOVIP.R8}\\
& \sum_{u \in N(v)} w_{vu} \leq ks_v & v \in V: \delta(v) > k \label{JOVIP.R9}\\
& s \in \{0,1\}^V,\ w \in \{0,1\}^{A_D},\ y \in \{0,1\}^{A_P}, x \in \{1,\ldots,|V|\}^V. \span \notag
\end{align}
}

The value of $M$ used in the computational experiments is not explicitly specified in \citet{Jovanovic2020}; therefore, we set $M = T$, as in \BRIIP.

\subsection{Formulation \FORTIP}

The last ILP formulation adapts the developments of \citet{Bozeman2019}.
In PDS, feasible solutions are simply sets $S \subseteq V$, called power dominating sets.
This formulation is motivated by a structural characterization of power dominating sets involving certain intersection conditions.

Let $G=(V,E)$ be a graph.
For a set $F \subseteq V$, the \emph{open-neighborhood} (resp. \emph{closed-neighborhood}) of $F$ is defined as $N(F) = (\bigcup_{v \in F} N(v)) \setminus F$ (resp. $N[F] = \bigcup_{v \in F} N[v]$).
When $F$ is nonempty and every vertex $v \in N(F)$ satisfies $|N(v) \cap F| \geq 2$, $F$ is a \emph{fort}.
Any power dominating set must contain at least one vertex in the closed-neighborhood of every fort $F$; otherwise, the vertices in $F$ cannot be monitored.
If $\FORT(G)$ denotes the family of all minimal forts of $G$, then
a set $S \subseteq V$ is a power dominating set of $G$ if and only if $S \cap N[F] \neq \emptyset$ for all $F \in \FORT(G)$.
Based on this characterization, the formulation of \citet{Bozeman2019} consists of a single exponential-size family of constraints:
\begin{align}
&\sum_{v \in N[F]} s_v \geq 1 & F \in \FORT(G).
\end{align}

For CPDS, we establish an analogous characterization based on forts for a $k$-capacitated power dominating function $\rho$.
This characterization is motivated by the fact that it is not sufficient for a vertex $u \in N(F)$ to belong to $S_{\rho}$ to monitor some vertex in $F$ through rule (DR); rather, $u$ must actually decide to monitor a vertex in $F$, that is, $\rho(u)$ must intersect $F$.

\begin{proposition}\label{prop.fort.cpds}
Let $\rho$ be a $k$-capacitated function.
Then, $\rho$ is a $k$-capacitated power dominating function if and only if, for every minimal fort $F \in \FORT(G)$, $S_{\rho} \cap F \neq \emptyset$ or there exists $u \in N(F) \cap S_{\rho}$ such that $\rho(u) \cap F \neq \emptyset$. 
\end{proposition}
\begin{proof}
Let $\rho$ be a $k$-capacitated power dominating function, and consider a proper calculation of $M(\rho) = V$.
For the sake of contradiction, suppose there exists $F \in \FORT(G)$ such that $S_{\rho} \cap F = \emptyset$ and, for every $u \in N(F) \cap S_{\rho}$, $\rho(u) \cap F = \emptyset$.

By its definition, a fort is nonempty, so $F \neq \emptyset$.
Let $v$ be the first vertex in $F$ that becomes monitored during the calculation of $M(\rho)$.
Under our assumptions, $v$ can only be monitored by applying rule (PR), say through a propagation $(u,v)$ for some $u \in N(v) \cap V_P$.

Since $v$ is the first vertex in $F$ to be monitored, it follows that $u \notin F$.
But then $u \in N(F)$, and the definition of fort ensures that $u$ must also be adjacent to some other vertex $w \in F \setminus \{v\}$.
Therefore, the propagation $(u,v)$ imposes the precedence $(w,v)$, but $v$ is monitored before $w$, which leads to a contradiction.

\medskip

To complete the proof, we demonstrate the converse implication.
Let $\rho$ fail to be a $k$-capacitated power dominating function, i.e., $M(\rho) \neq V$.
Define $F = V \setminus M(\rho) \neq \emptyset$.
We first show that $F$ is a fort.

If $N(F) = \emptyset$, then $F$ is trivially a fort.
Otherwise, let $v \in N(F)$.
We claim that $v$ must actually have more than one neighbor in $F$; otherwise, an outgoing propagation from $v$ could be applied to increase the size of $M(\rho)$, since $v$ and all of its remaining neighbors (if any) lie in the complement of $F$ and are therefore already monitored.
Hence, $F$ is a fort.

Clearly, $F$ contains some minimal fort $F'$.
Next, $F'$ must satisfy $S_{\rho} \cap F' = \emptyset$ and, for every $u \in N(F') \cap S_{\rho}$, we must have $\rho(u) \cap F' = \emptyset$; otherwise, $F'$ would contain a vertex that could be monitored by rule (DR), and the size of $M(\rho)$ could again be increased.
\end{proof}

This result allows us to propose the following formulation for CPDS, denoted by \FORTIP. Constraints \eqref{FORTIP.R1} are a direct translation of Proposition \ref{prop.fort.cpds}, and constraints \eqref{FORTIP.R2} are equivalent to \eqref{FPSIP.R2} in \FPSIP.
{\small
\begin{align}
min\  & \sum_{v \in V} s_v \label{FORTIP.ObjFun}\\ 
s.t.\ & \sum_{v \in F} s_v + \sum_{\mathclap{\substack{v \in N(F):\\ \delta(v) \leq k}}} s_v + \sum_{\mathclap{\substack{v \in N(F):\\ \delta(v) > k }}} \hspace{10pt} \sum_{u \in N(v) \cap F} w_{vu} \geq 1 & F \in \FORT(G) \label{FORTIP.R1} \\
& \sum_{u \in N(v)} w_{vu} \leq ks_v & v \in V: \delta(v) > k \label{FORTIP.R2}\\
& s \in \{0,1\}^V,\ w \in \{0,1\}^{A_D}. \span \notag
\end{align}
}

To solve this formulation, we adapt the lazy-constraint separation of \citet{Blasius2024} for PDS to dynamically generate constraints \eqref{FORTIP.R1}, with adjustments to account for limited capacity.
As noted by the authors, to speed up this separation it is crucial to avoid computing each monitored set from scratch; instead, we apply local changes to previous computations to quickly unmonitor or monitor vertices.

\medskip

Other ILP formulations in the PDS literature were not considered for comparison for the following reasons. 
The formulation by \citet{Aazami2010} was corrected and evaluated by \citet{Brimkov2019}, who concluded that Aazami's model is significantly outperformed.
The single-level ILP formulation by \citet{Carvalho2018} is only able to solve instances with up to 150 and 300 vertices, which is considerably smaller than the instances considered in this work.
Some additional models without correctness proofs were also omitted.
 
\section{Computational experiments}\label{sec.experiments}

In this section, we conduct computational experiments on 26 standard power network benchmark instances, ranging from 200 to approximately 14.000 vertices, all extracted from the open-source tools pandapower \citep{PANDAPOWER} and matpower \citep{MATPOWER}.
The names of the considered power networks are shown on the x-axis of Figure \ref{fig.FPSIP.vs.EFPSIPS} and usually contain the number of vertices as a substring, except for Texas and Western, which have 2000 and 10024 vertices, respectively.
For each power network, we define a CPDS instance consisting of its underlying graph, its set of zero-injection vertices, and a capacity parameter $k$.
We vary $k$ from 0 to a value $k^*$, which is the minimum capacity for which CPDS and PDS share the same optimal value.
Note that $k^*$ may be lower than the maximum degree of the graph.
As a result, a total of 262 instances are considered. 

We use Gurobi 12.0.3 \citep{GUROBI} to solve ILP, called through the \texttt{C++20} API.
Formulations \BRIIP and \JOVIP are solved straightforwardly as black-box models, whereas user-defined lazy-constraint callbacks are provided for \FPSIP, \EFPSIP, and \FORTIP.
This approach differs from \citet{Blasius2024}, who use Gurobi only to solve linear relaxations, while the remaining branch-and-bound tree is handled separately.
In addition, we do not incorporate their decomposition and reduction techniques.

We used a machine equipped with Ubuntu 22.04.5 LTS, 8 GB of RAM, and an Intel(R) Core(TM) i7-9700 CPU at 3.00GHz with 8 threads (all available).
Gurobi is used with default parameters, except for a time limit of 900 s per instance.
Five \emph{runs} (executions) are performed for each instance and model in order to mitigate variability in computation times.

A concise summary of the computational results is reported in Table \ref{tab.summary}.
Each row reports values averaged over all instances and runs, aggregating a total of 1310 observations.
The table reports, for each model, the number of variables, the number of constraints initially in the formulation and the number of lazy-constraints added, the percentage of runs that reach optimality, referred to as \emph{optimality rate}, the execution time (assigning 900 s to runs that hit the time limit), the cumulative time spent in lazy-constraint separation, and the final relative gap (averaged over runs that reached the time limit, using the number of vertices as an upper bound when no feasible solution was found within the time limit).
Models \EFPSIP-InP, \EFPSIP-OutP, and \EFPSIP-InP-OutP incorporate into \EFPSIP the incoming and outgoing propagation constraints (InP) and (OutP) from Section \ref{sec.incoming.outgoing.constraints}, individually and jointly, while \EFPSIP-OutP-Init additionally includes the initialization of EFPS constraints for 2-cycles described in Section \ref{sec.constraint.generation}.
The results reported in each row will be discussed incrementally in detail as each experiment is introduced.

Our code and the full table of results summarized in Table \ref{tab.summary} are available online at: \url{https://github.com/maurolucci/cpds}.

\begin{table}[]
\centering \footnotesize
\renewcommand{\arraystretch}{1.1}
\setlength{\tabcolsep}{4pt}
\begin{tabular}{lccccccc}
\toprule
& & \multicolumn{2}{c}{\textbf{\#Constraints}} & & \multicolumn{2}{c}{\textbf{Time}} & \\
\cmidrule(l{2pt}r{2pt}){3-4} \cmidrule(l{2pt}r{2pt}){6-7}
& \textbf{\#Variables} & Initial & Lazy & \textbf{Optimality} & Total & Lazy & \textbf{Gap} \\
\textbf{Model} & ($\times 10^3$) & ($\times 10^3$) & & (\%) & (s) & (s) & (\%) \\
\midrule
\FPSIP & 13.2 & 5.6 & 5461 & 90.2 & 109.9 & 0.17 & 0.78 \\
\EFPSIP & 13.2 & 5.6 & 5185 & 90.9 & 99.9 & 0.15 & 0.19 \\
\EFPSIP-InP & 13.2 & 8.5 & 3207 & 90.5 & 104.3 & 0.16 & 0.20 \\
\EFPSIP-OutP & 13.2 & 7.0 & 134 & 94.1 & 61.0 & 0.035 & 0.17 \\
\EFPSIP-InP-OutP & 13.2 & 10.0 & 128 & 94.0 & 63.2 & 0.035 & 0.19 \\
\EFPSIP-OutP-Init & 13.2 & 15.8 & 17 & \textbf{94.4} & \textbf{60.8} & 0.021 & 0.18 \\
\midrule
\BRIIP & 17.8 & 26.5 & -- & 92.4 & 103.3 & -- & 0.24 \\
\JOVIP & 17.8 & 51.4 & -- & 92.0 & 141.5 & -- & 0.20 \\
\FORTIP & 8.5 & 1.0 & 22096 & 53.5 & 431.7 & 1.0 & 27.81 \\
\bottomrule
\end{tabular}
\caption{Summary of the average computational results for each model.}
\label{tab.summary}
\end{table}

\subsection{\FPSIP vs. \EFPSIP}

The first experiment compares \FPSIP and \EFPSIP.
From Table \ref{tab.summary}, \FPSIP solves 1182 out of 1310 runs to optimality (90.2 \%), which is slightly improved by \EFPSIP, with 1191 (\mbox{90.9 \%}).
The average execution time is 109.9 s for \FPSIP, compared to 99.9 s for \EFPSIP, corresponding to a 9.1 \% reduction.
The average gap is 0.78 \% for \FPSIP and 0.19 \% for \EFPSIP, representing a 75.6 \% improvement.

Figure \ref{fig.FPSIP.vs.EFPSIPS.time} reports the average execution times  of the models as a function of the number of vertices in each power network, displayed on a logarithmic scale; each marker corresponds to the average over all capacities and runs for a given network and model.
As expected, execution times tend to increase with the number of vertices.
Despite minor differences, both models display similar trends, possibly as a consequence of averaging over different capacities.
Therefore, we provide additional plots with a finer level of disaggregation.

Figure \ref{fig.FPSIP.vs.EFPSIPS.best-performing} displays, for each instance (power network and capacity), a marker whose color and shape indicate the best-performing model.
For a given instance, a model is considered best if it solves the largest number of runs to optimality across the five runs; ties are resolved, whenever possible, by the lowest average execution time, followed by the lowest average final upper bound and then the lowest average final lower bound.
When ties cannot be resolved, the instance is labeled as a ``tie''.
In addition, the marker size encodes the optimality rate of the best-performing model across those five runs, while the rightmost bar shows the percentage of instances won by each model.
Instances with large networks and capacities around $k=2$ appear to be the most challenging, and overall \EFPSIP emerges as the best-performing model in 72.1 \% of the instances, whereas \FPSIP is the best in the remaining 27.9 \%.

\begin{figure}
\begin{subfigure}{0.47\textwidth}
\includegraphics[width=\linewidth]{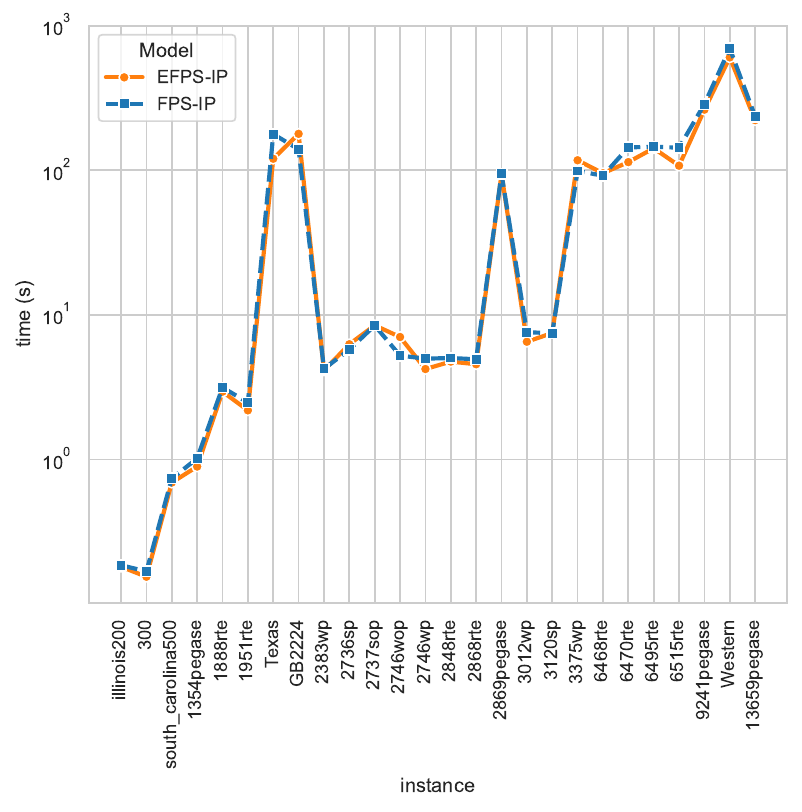}
\caption{Average execution time over all capacities by power network}
\label{fig.FPSIP.vs.EFPSIPS.time}
\end{subfigure}
\hfill
\begin{subfigure}{0.49\textwidth}
\includegraphics[width=\linewidth]{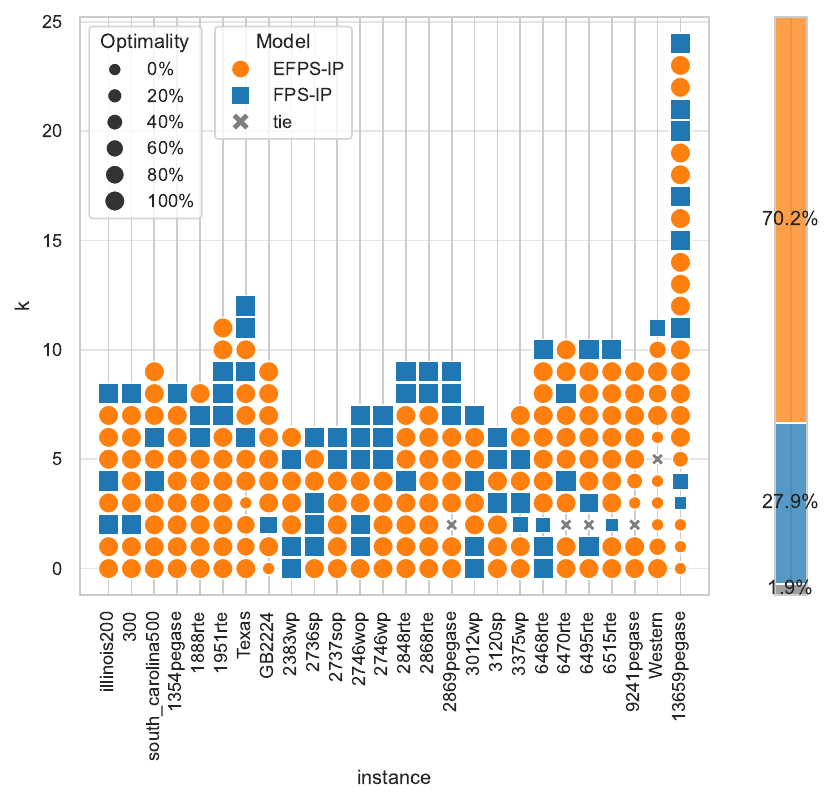}
\caption{Best-performing model by instance (power network and capacity)}
\label{fig.FPSIP.vs.EFPSIPS.best-performing}
\end{subfigure}

\caption{Results of the comparison between \FPSIP and \EFPSIP}
\label{fig.FPSIP.vs.EFPSIPS}

\end{figure}

Finally, each plot in Figure \ref{fig.FPS.vs.EFPS.grid} depicts, for a given network, the evolution of execution time across the different capacities, averaged over the five runs.
For space reasons, only approximately half of the networks are displayed, selected in an alternating manner.
Most networks exhibit a common pattern, with a peak at $k=2$ followed by decreasing times for larger capacities.
At this level of disaggregation, narrow differences between \FPSIP and \EFPSIP emerge: in most cases, \EFPSIP achieves lower execution times, although some exceptions can be observed.
These results motivate our choice to focus on \EFPSIP in the remainder of this section.

\begin{figure}
    \centering
    \includegraphics[width=\linewidth]{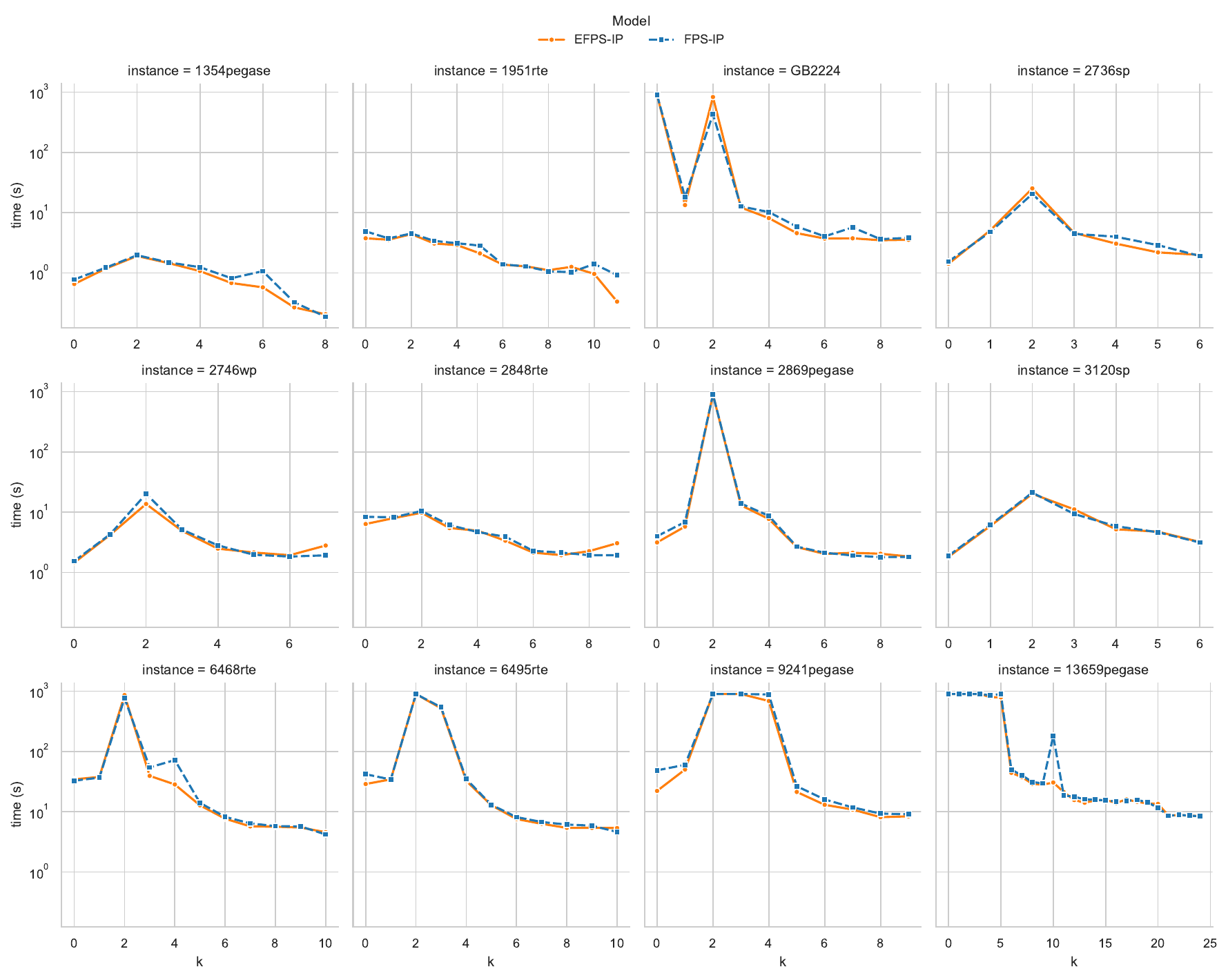}
    \caption{Evolution of the execution time for \EFPSIP and \EFPSIP with respect to capacity for selected power networks.}
    \label{fig.FPS.vs.EFPS.grid}
\end{figure}

\subsection{\EFPSIP: model tightening and initial constraints}

The second experiment focuses on evaluating the impact of the potential improvements for \EFPSIP discussed throughout this work: specifically, the incorporation of valid inequalities, redundancy-breaking constraints, and the initialization of the lazy-constraint generation framework.

We first analyze the effect of adding the incoming and outgoing propagation constraints, both individually and jointly.
From Table \ref{tab.summary}, incorporating InP contraints appears to be counterproductive, as both the optimality rate and the execution time deteriorate.
In contrast, the results obtained with OutP constraints stand out.
Compared with \EFPSIP, the number of runs that reach optimality increases from 1191 (90.9 \%) to 1233 (94.1 \%), corresponding to a 3.5 \% improvement, while the execution time decreases from 99.9 s to 61.0 s, yielding a 38.9 \% improvement.
Moreover, a noteworthy effect on lazy-separation performance is observed when these constraints are introduced.
Both InP and OutP constraints substantially reduce the need for separation, as the average number of lazy constraints added decreases from 5185 to 3207 with InP constraints, and to only 134 with OutP constraints.
Overall, the cumulative time spent in the lazy-separation routine is negligible, remaining below 0.17 second on average.
Combining InP and OutP constraints does not provide any additional benefit beyond that obtained with OutP alone.

These results are consistent with the plot in Figure \ref{fig.Model.Tightening},
where two distinct trends appear depending on whether OutP constraints are included.

The impact of initialization is now evaluated on the best-performing model, namely \EFPSIP-OutP.
Up to this point, an empty set of EFPS constraints was initially considered, and we now propose an initialization that includes all EFPS constraints associated with any 2-cycles in the precedence digraph.
Compared with \EFPSIP-OutP, the results reported in Table \ref{tab.summary} show subtle enhancements in terms of the optimality rate and execution time when initialization is used.
Besides, the number of initial constraints approximately doubles and the number of lazy constraints added is reduced by nearly a factor of eight.
Additionally, when execution times are disaggregated by network, modest time reductions emerge for some of them, as illustrated in Figure \ref{fig.Initial.Constraints}.


\begin{figure}

\begin{subfigure}{0.48\textwidth}
\includegraphics[width=\linewidth]{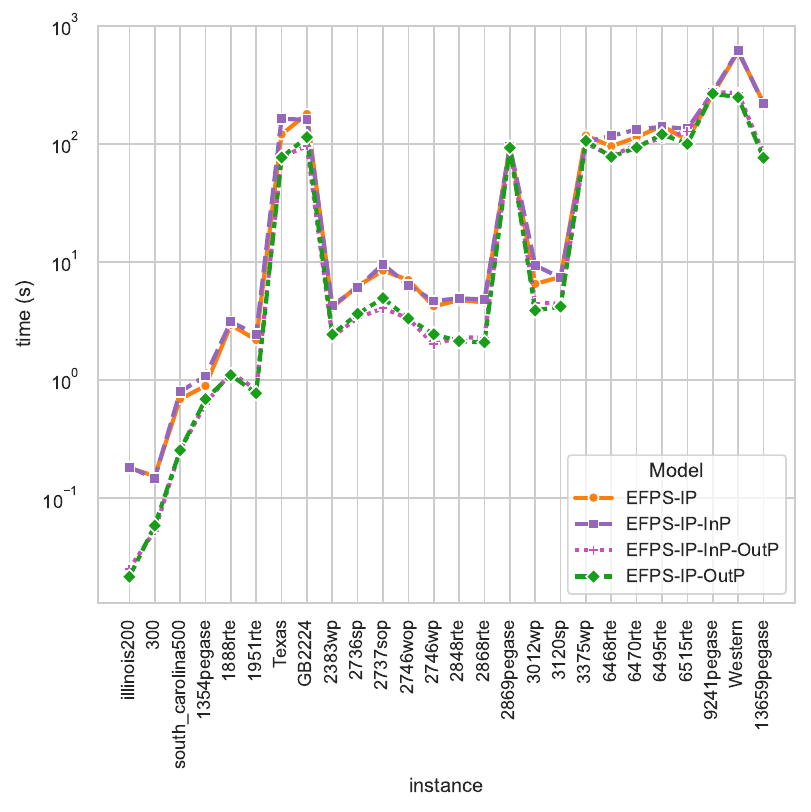}
\caption{Average execution time by power network for \EFPSIP with and without constraints (OutP) and (InP)}
\label{fig.Model.Tightening}
\end{subfigure}
\hfill
\begin{subfigure}{0.48\textwidth}
\includegraphics[width=\linewidth]{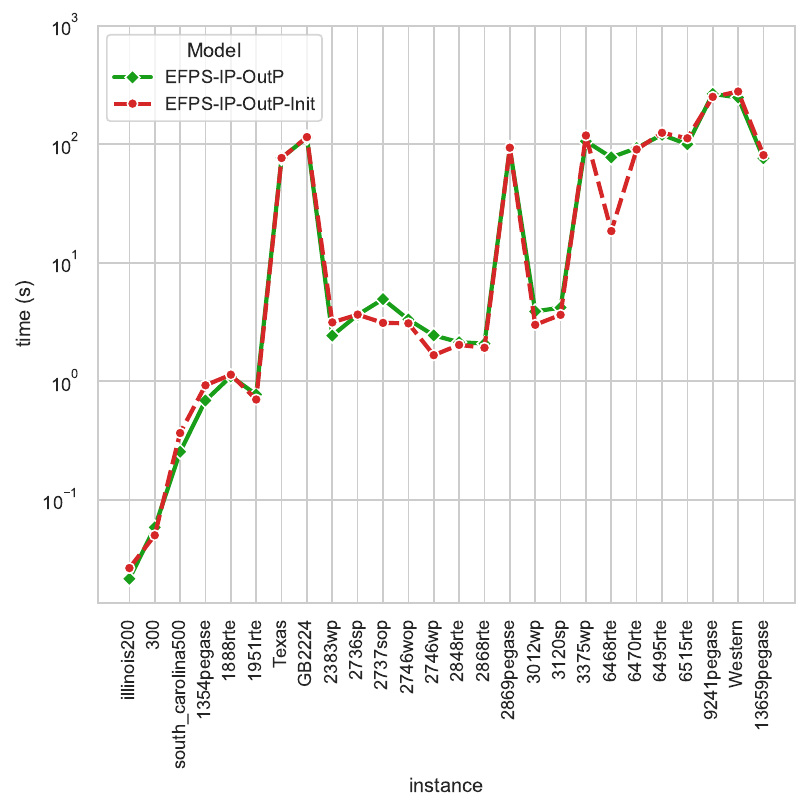}
\caption{Average execution time by power network for \EFPSIP-OutP with and without initial constraints}
\label{fig.Initial.Constraints}
\end{subfigure}

\caption{Results of the experiment with model tightening and initial constraints}
\label{fig.Model.Tightening.Initial.Constraints}
\end{figure}

\subsection{\EFPSIP vs. models adapted from the literature}

We finally compare our developments with the ILP formulations adapted from the literature.
We consider the best-performing model from the previous experiment, namely \EFPSIP-OutP-Init, together with the models introduced in Section \ref{sec.formulations.from.literature}, i.e. \BRIIP, \JOVIP, and \FORTIP.

According to the summary results in Table \ref{tab.summary}, \EFPSIP-OutP-Init achieves the highest optimality rate and the lowest execution time: 94.4 \% and \mbox{60.8 s}.
It is followed by \BRIIP, which achieves an optimality rate of 92.4 \% and an average execution time of 103.3 s; this implies that our approach is 69.9 \% faster on average, corresponding to an average speedup of 1.7x.
Next comes \JOVIP, with 92.0 \% and 141.5 s.
By a large margin, \FORTIP ranks last, with only \mbox{53.5 \%} and 431.7 s.
Notably, these time differences persist when averaging only over runs that reach optimality: \mbox{11.3 s} for \EFPSIP-OutP-Init, 35.2 s for \FORTIP, 37.5 s for \BRIIP, and 75.4 for \JOVIP.
Additionally, the average relative gap over runs that hit the time limit is \mbox{27.81 \%} for \FORTIP, while it is very small for the remaining models, ranging from 0.18 \% to 0.24 \%.

Figure \ref{fig.EFPS.vs.Literature.cumulative.time} and Figure \ref{fig.EFPS.vs.Literature.cumulative.gap} show the cumulative percentage of solved runs as a function of execution time and relative gap, respectively.
Figure \ref{fig.EFPS.vs.Literature.time} shows the average execution time, disaggregated by network size.
The models \EFPSIP-OutP-Init, \BRIIP, and \JOVIP exhibit comparable performance, with \EFPSIP-OutP-Init consistently outperforming the others.
The model \FORTIP struggles on most networks instead.
Figure \ref{fig.EFPS.vs.Literature.best-performing} illustrates the best-performing model for each instance, revealing a clear superiority of \EFPSIP-OutP-Init, which is the best-performing model in 90.8 \% of the instances.
For tied instances, \EFPSIP-OutP-Init and \BRIIP achieve identical performance.

\begin{figure}
\begin{subfigure}{0.48\textwidth}
\includegraphics[width=\linewidth]{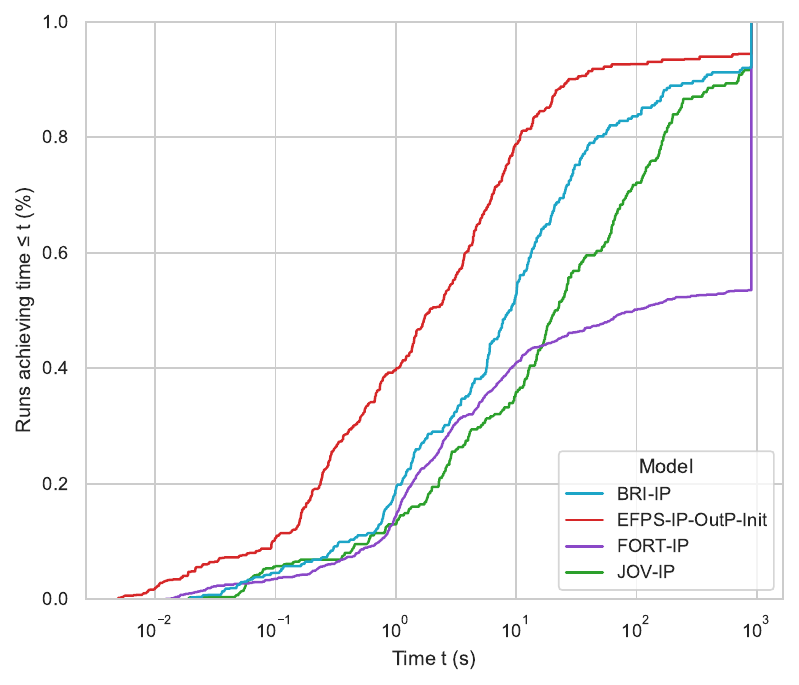}
\caption{Cumulative percentage of runs solved within $t$ seconds.}
\label{fig.EFPS.vs.Literature.cumulative.time}
\end{subfigure}
\hfill
\begin{subfigure}{0.49\textwidth}
\includegraphics[width=\linewidth]{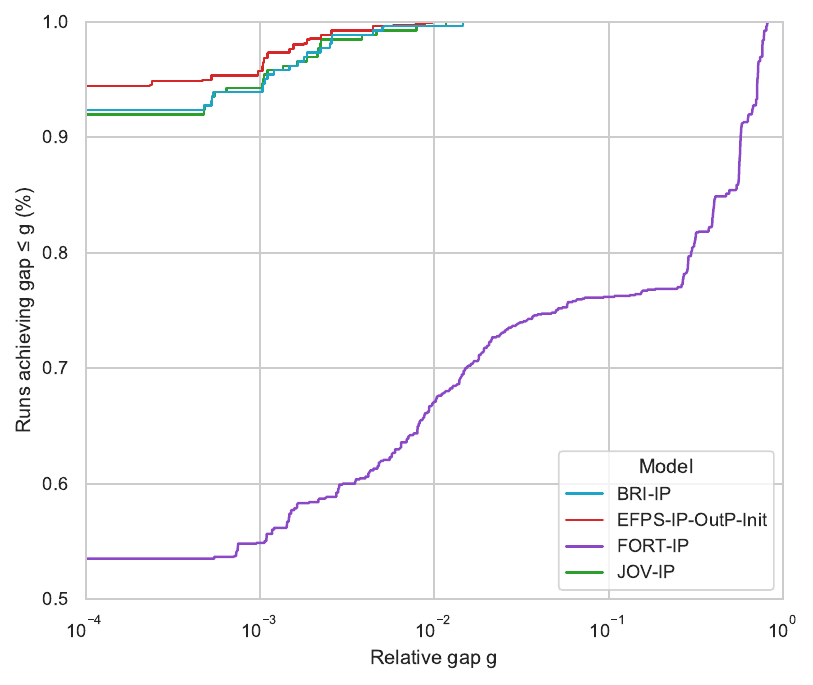} 
\caption{Cumulative percentage of runs with relative gap at most $g$.}
\label{fig.EFPS.vs.Literature.cumulative.gap}
\end{subfigure}

\begin{subfigure}{0.47\textwidth}
\includegraphics[width=\linewidth]{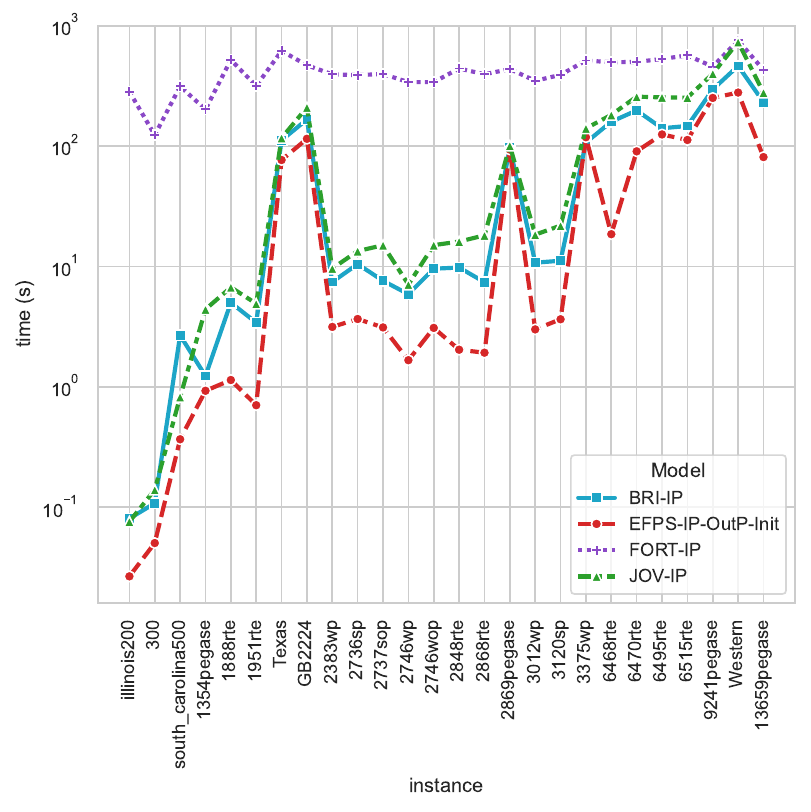}
\caption{Average execution time over all capacities by power network}
\label{fig.EFPS.vs.Literature.time}
\end{subfigure}
\hfill
\begin{subfigure}{0.49\textwidth}
\includegraphics[width=\linewidth]{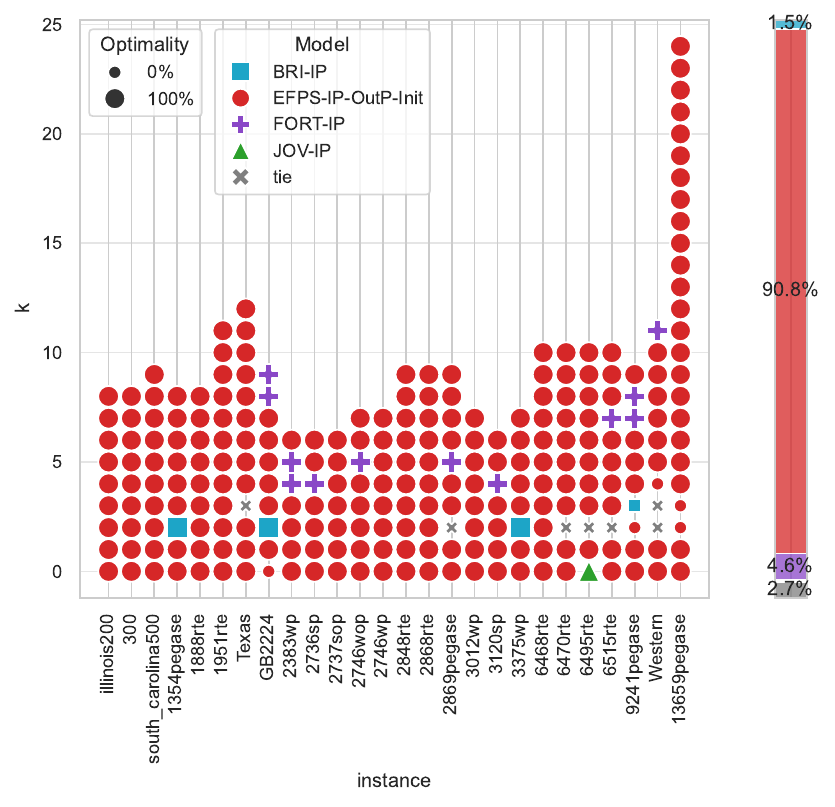}
\caption{Best-performing model by instance (power network and capacity)}
\label{fig.EFPS.vs.Literature.best-performing}
\end{subfigure}

\caption{Comparison between \EFPSIP and existing models from the literature}
\label{fig.EFPS.vs.Literature}
\end{figure}

Finally, Figure \ref{fig.EFPS.vs.Literature.Grid} extends Figure \ref{fig.EFPS.vs.Literature.time} by presenting execution time as a function of capacity for some power networks.
\EFPSIP-OutP-Init is consistently faster, with speedups of up to an order of magnitude for certain capacities and networks.
Peaks at $k = 2$ are observed for all models except \FORTIP, which instead appears to follow a decreasing trend as capacity increases, becoming competitive toward higher values of $k$.

\begin{figure}
    \centering
    \includegraphics[width=\linewidth]{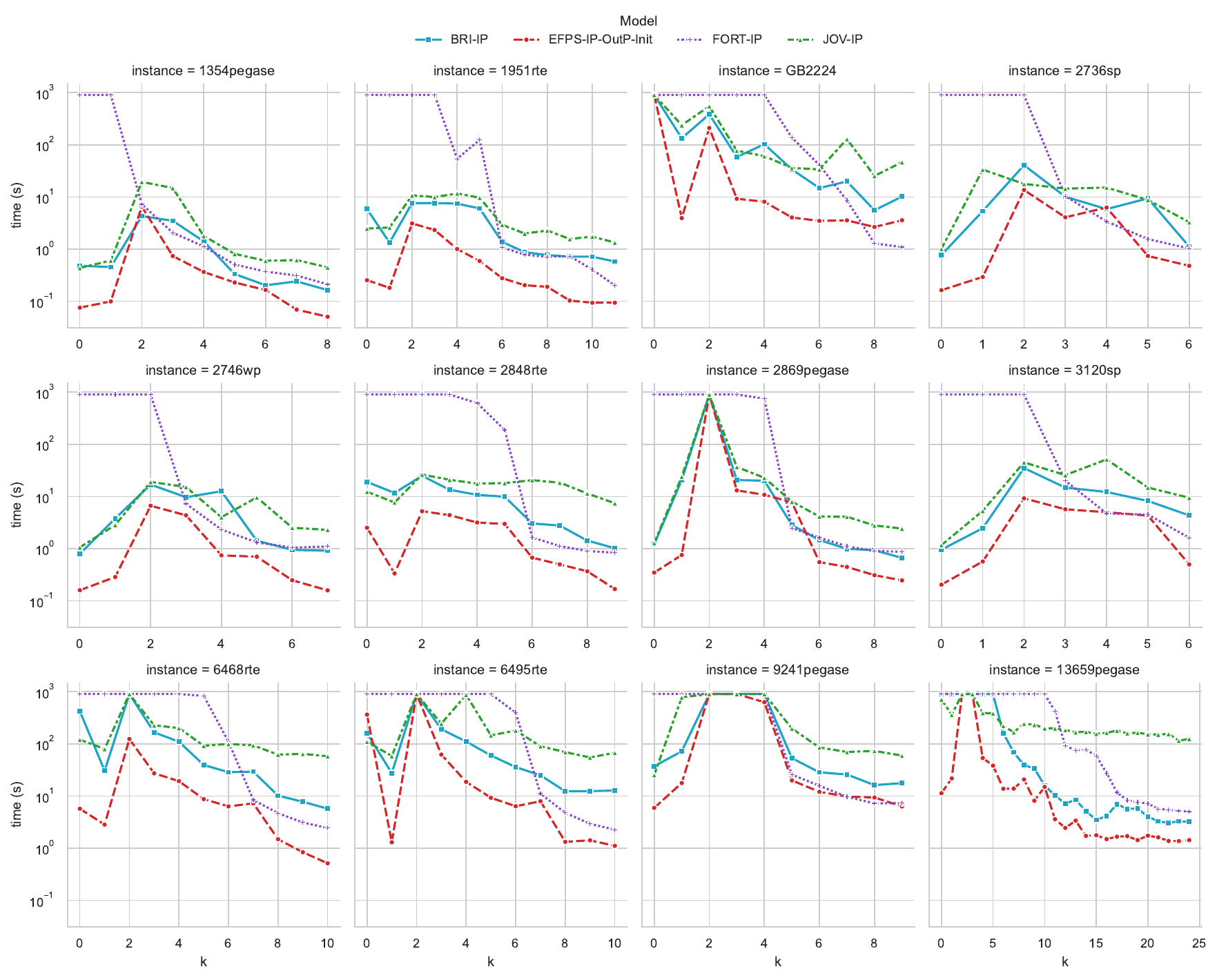}
    \caption{Evolution of the execution time for \EFPSIP-OutP-Init and existing models from the literature for selected power networks.}
    \label{fig.EFPS.vs.Literature.Grid}
\end{figure}

\section{Conclusions and future work}\label{sec.conclusions.future.work}

This work addresses a capacitated variant of the power dominating set problem, which is more consistent with real-world applications in power network monitoring.
In this variant, the common assumption that a vertex can monitor all of its neighbors is relaxed by introducing capacities, reflecting the limited number of channels available in PMUs due to manufacturing decisions.

We introduce the novel concept of forbidden propagation sets (FPSs), which are sets of propagations that cannot be applied together because cyclic temporal precedences are created.
These structures are the core of this work and give rise to new ILP formulations.
Our formulations share some of the standard variables of infection-based models, but do not rely on any big-$M$ constraints. 
However, as in fort-based models, the number of FPSs may grow exponentially; consequently, the resulting models are solved using a lazy-constraint generation approach.
Through an in-depth structural study, we characterize minimal FPSs, which leads to an efficient separation procedure that dynamically generates violated constraints via cycle detection in a particular digraph.

The computational experiments on benchmark power networks with up to 14,000 vertices yield promising results.
The proposed approach exhibits outstanding performance, outperforming the adaptation of \citet{Brimkov2019} by an average factor of 1.70 in runtime,  \citet{Jovanovic2020} by a factor of 2.33, and \citet{Blasius2024} by a factor of 7.10; under the setting of our computational experiments.
We observe that performance depends not only on the size of the power network, but also on the capacity.
In particular, for $k=2$ all approaches exhibit significant difficulties; also observed by \citet{Carvalho2018}, but under a slightly different propagation rule.
Moreover, the current state of the art for solving the power dominating set problem, which is based on forts, struggles in the presence of capacities.
Indeed, the performance of the fort-based model deteriorates as the capacities decrease.

Based on these results, we conducted an additional computational experiment to evaluate the performance of \BRIIP when incorporating OutP constraints.
Although these constraints lead to some performance improvement, the gains are rather limited, since our approach still achieves an average speedup of 1.36x in execution time and outperforms it on 92.7 \% of the instances.

We believe that the FPS-based approach can be further improved through a deeper structural analysis aimed at identifying valid inequalities that strengthen the models.
We also experimented with the use of EFPS constraints to cut fractional solutions.
The preliminary results were encouraging, but this approach requires further investigation to become practically effective.

Another promising direction is to combine FPS-based models with decomposition techniques and reduction rules.
In this regard, it would be worthwhile to study whether the methods of \citet{Blasius2024} for the power dominating set problem can be adapted to the capacitated variant.
It would also be interesting to extend the concept of FPSs to related variants, such as problems with heterogeneous capacities or objective functions that depend on capacities.
Finally, alternative FPS-based models with variables associated with subsets of at most $k$ neighbors appear promising, representing vertices monitored through the application of the domination rule; see \citet{Korkali2009}.
Solving the resulting formulations would require not only row (constraint) generation, but also column generation to handle the exponential number of variables, embedded within a branch-and-price framework.



\bibliographystyle{model5-names} 
\bibliography{bibliography}

@article{Baldwin1993,
author={Baldwin, T.L. and Mili, L. and Boisen, M.B. and Adapa, R.},
journal={IEEE Transactions on Power Systems}, 
title={Power system observability with minimal phasor measurement placement}, 
year={1993},
volume={8},
number={2},
pages={707-715},
doi={10.1109/59.260810}
}

@article{Haynes2002,
author = {Haynes, Teresa W. and Hedetniemi, Sandra M. and Hedetniemi, Stephen T. and Henning, Michael A.},
title = {Domination in Graphs Applied to Electric Power Networks},
journal = {SIAM Journal on Discrete Mathematics},
volume = {15},
number = {4},
pages = {519-529},
year = {2002},
doi={10.1137/S0895480100375831}
}

@article{Brueni2005,
author = {Brueni, Dennis J. and Heath, Lenwood S.},
title = {The PMU Placement Problem},
journal = {SIAM Journal on Discrete Mathematics},
volume = {19},
number = {3},
pages = {744-761},
year = {2005},
doi={10.1137/S0895480103432556}
}

@article{Guo2008,
author={Guo, Jiong and Niedermeier, Rolf and Raible, Daniel},
title={Improved Algorithms and Complexity Results for Power Domination in Graphs},
journal={Algorithmica},
year={2008},
month={Oct},
day={01},
volume={52},
number={2},
pages={177-202},
issn={1432-0541},
doi={10.1007/s00453-007-9147-x}
}

@article{Xu2006,
title = {Power domination in block graphs},
journal = {Theoretical Computer Science},
volume = {359},
number = {1},
pages = {299-305},
year = {2006},
issn = {0304-3975},
author = {Guangjun Xu and Liying Kang and Erfang Shan and Min Zhao},
doi={10.1016/j.tcs.2006.04.011}
}

@article{Liao2013,
author={Liao, Chung-Shou
and Lee, D. T.},
title={Power Domination in Circular-Arc Graphs},
journal={Algorithmica},
year={2013},
month={Feb},
day={01},
volume={65},
number={2},
pages={443-466},
issn={1432-0541},
doi={10.1007/s00453-011-9599-x}
}

@article{Binkele-Raible2012,
author={Binkele-Raible, Daniel
and Fernau, Henning},
title={An Exact Exponential Time Algorithm for Power Dominating Set},
journal={Algorithmica},
year={2012},
month={Jun},
day={01},
volume={63},
number={1},
pages={323-346},
issn={1432-0541},
doi={10.1007/s00453-011-9533-2}
}

@article{Aazami2010,
author={Aazami, Ashkan},
title={Domination in graphs with bounded propagation: algorithms, formulations and hardness results},
journal={Journal of Combinatorial Optimization},
year={2010},
month={May},
day={01},
volume={19},
number={4},
pages={429-456},
issn={1573-2886},
doi={10.1007/s10878-008-9176-7}
}

@inproceedings{Fan2012,
author="Fan, Neng
and Watson, Jean-Paul",
editor="Lin, Guohui",
title="Solving the Connected Dominating Set Problem and Power Dominating Set Problem by Integer Programming",
booktitle="Proceedings of the International Conference on Combinatorial Optimization and Applications",
series="Lecture Notes in Computer Science",
volume="7402",
year="2012",
publisher="Springer",
pages="371--383",
note={{Banff, Canada}},
isbn="978-3-642-31770-5",
doi={10.1007/978-3-642-31770-5_33}
}

@article{Bozeman2019,
author={Bozeman, Chassidy
and Brimkov, Boris
and Erickson, Craig
and Ferrero, Daniela
and Flagg, Mary
and Hogben, Leslie},
title={Restricted power domination and zero forcing problems},
journal={Journal of Combinatorial Optimization},
year={2019},
month={Apr},
day={01},
volume={37},
number={3},
pages={935-956},
issn={1573-2886},
doi={10.1007/s10878-018-0330-6}
}

@article{Brimkov2019,
author={Brimkov, Boris
and Mikesell, Derek
and Smith, Logan},
title={Connected power domination in graphs},
journal={Journal of Combinatorial Optimization},
year={2019},
month={Jul},
day={01},
volume={38},
number={1},
pages={292-315},
issn={1573-2886},
doi={10.1007/s10878-019-00380-7}
}

@article{Jovanovic2020,
author = {Jovanovic, Raka and Voss, Stefan},
title = {The fixed set search applied to the power dominating set problem},
journal = {Expert Systems},
volume = {37},
number = {6},
pages = {e12559},
year = {2020},
doi={10.1111/exsy.12559}
}

@article{Blasius2024,
author = {Bl\"{a}sius, Thomas and G\"{o}ttlicher, Max},
title = {An Efficient Algorithm for Power Dominating Set},
year = {2024},
issue_date = {Mar 2025},
publisher = {Springer-Verlag},
address = {Berlin, Heidelberg},
volume = {87},
number = {3},
issn = {0178-4617},
journal = {Algorithmica},
month = dec,
pages = {344–376},
numpages = {33},
doi={10.1007/s00453-024-01283-8}
}

@article{Carvalho2018,
title = {Observability of power systems with optimal PMU placement},
journal = {Computers \& Operations Research},
volume = {96},
pages = {330-349},
year = {2018},
issn = {0305-0548},
author = {Margarida Carvalho and Xenia Klimentova and Ana Viana},
doi={10.1016/j.cor.2017.10.012}
}

@inproceedings{Korkali2009,
author={Korkali, Mert and Abur, Ali},
booktitle={Proceedings of the IEEE Power \& Energy Society General Meeting}, 
title={Placement of PMUs with channel limits}, 
year={2009},
volume={},
number={},
pages={1-4},
note={{Calgary, Canada}},
doi={10.1109/PES.2009.5275529}
}

@article{Joshi2021,
title = {Synchrophasor measurement applications and optimal PMU placement: A review},
journal = {Electric Power Systems Research},
volume = {199},
pages = {107428},
year = {2021},
issn = {0378-7796},
author = {Prachi Mafidar Joshi and H.K. Verma},
doi={10.1016/j.epsr.2021.107428}
}

@article{Almasabi2019,
author={Almasabi, Saleh and Mitra, Joydeep},
journal={IEEE Transactions on Smart Grid}, 
title={A Fault-Tolerance Based Approach to Optimal PMU Placement}, 
year={2019},
volume={10},
number={6},
pages={6070-6079},
doi={10.1109/TSG.2019.2896211}
}

@ARTICLE{MATPOWER,
  author={Zimmerman, Ray Daniel and Murillo-Sánchez, Carlos Edmundo and Thomas, Robert John},
  journal={IEEE Transactions on Power Systems}, 
  title={MATPOWER: Steady-State Operations, Planning, and Analysis Tools for Power Systems Research and Education}, 
  year={2011},
  volume={26},
  number={1},
  pages={12-19},
  doi={10.1109/TPWRS.2010.2051168}
}

@article{PANDAPOWER,
author={L. Thurner and A. Scheidler and F. Schafer and J. H. Menke and J. Dollichon and F. Meier and S. Meinecke and M. Braun},
journal={IEEE Transactions on Power Systems},
title={pandapower - an Open Source Python Tool for Convenient Modeling, Analysis and Optimization of Electric Power Systems},
year={2018},
volume={33},
number={6},
pages={6510-6521},
ISSN={0885-8950},
doi={10.1109/TPWRS.2018.2829021}
}

@misc{GUROBI,
author = {{Gurobi Optimization, LLC}},
title = {{Gurobi Optimizer Reference Manual}},
year = 2026,
url = "https://www.gurobi.com"
}

@article{Lucci2025,
title = {Integer linear programs for the power dominating set problem with channel limitation},
journal = {Procedia Computer Science},
volume = {273},
pages = {54-61},
year = {2025},
issn = {1877-0509},
author = {Mauro Lucci and Diego {Delle Donne} and Mariana Escalante},
doi={10.1016/j.procs.2025.10.280}
}

\end{document}